\documentclass[a4paper,10pt]{article}

\usepackage{amsmath}
\usepackage{amsthm}
\usepackage{amssymb}
\usepackage{graphicx}
\usepackage{bbm}
\usepackage{mathrsfs}
\usepackage[center]{caption}
\usepackage{enumerate}
\usepackage{verbatim}
\usepackage{authblk}
\usepackage[hmargin=3cm, vmargin=3.5cm]{geometry}
\usepackage{tikz}

\theoremstyle{plain}
\newtheorem{thm}{Theorem}[]
\newtheorem{cor}[thm]{Corollary}
\newtheorem{lem}[thm]{Lemma}
\newtheorem{prop}[thm]{Proposition}

\theoremstyle{definition}

\newcommand{\Z}{\mathbb{Z}}
\renewcommand{\P}{\mathbb{P}}
\newcommand{\E}{\mathbb{E}}
\newcommand{\R}{\mathbb{R}}
\newcommand{\F}{\mathcal{F}}

\newcommand{\ind}{\mathbbm{1}}

\newcommand{\eps}{\varepsilon}
\newcommand{\bp}{\begin{proof}}
\newcommand{\ep}{\end{proof}}
\newcommand{\Ecal}{\mathcal{E}}

\newcommand{\I}{\mathcal{I}}
\renewcommand{\d}{\mathop{}\!\mathrm{d}}
\def\bal#1\eal{\begin{align*}#1\end{align*}}
\newcommand{\N}{\mathbb{N}}
\newcommand{\odd}{E^{\text{odd}}_n(t)}
\newcommand{\even}{E^{\text{even}}_n(t)}

\author{Martin Prigent\thanks{University of Bath, Department of Mathematical Sciences, Bath BA2 7AY, UK. \texttt{martinprigent@outlook.com}}\hspace{1.5mm} and Matthew I.~Roberts\thanks{University of Bath, Department of Mathematical Sciences, Bath BA2 7AY, UK. \texttt{mattiroberts@gmail.com}}}

\title{Noise sensitivity and exceptional times of transience for a simple symmetric random walk in one dimension}

\begin{document}
\maketitle

\begin{abstract}
We define a dynamical simple symmetric random walk in one dimension, and show that there almost surely exist exceptional times at which the walk tends to infinity. This is in contrast to the usual dynamical simple symmetric random walk in one dimension, for which such exceptional times are known not to exist. In fact we show that the set of exceptional times has Hausdorff dimension $1/2$ almost surely, and give bounds on the rate at which the walk diverges at such times.

We also show noise sensitivity of the event that our random walk is positive after $n$ steps. In fact this event is maximally noise sensitive, in the sense that it is quantitatively noise sensitive for any sequence $\eps_n$ such that $n\eps_n\to\infty$. This is again in contrast to the usual random walk, for which the corresponding event is known to be noise stable.
\end{abstract}

\section{Introduction and results}
Consider two simple symmetric random walks in one dimension. The first, at each step independently, jumps upwards with probability $1/2$ or downwards with probability $1/2$. The second begins facing upwards and, at each step independently, decides to keep moving the same way with probability $1/2$ or switches direction with probability $1/2$.

We call the first of these two random walks the \emph{compass} random walk, as it has an in-built sense of direction, and the second the \emph{switch} random walk, as it only decides whether or not to switch directions. Of course these two random walks have exactly the same distribution---they are simple symmetric random walks---although, as we will see when we define them rigorously, they are different functions of the underlying randomness. This means that when we talk about noise sensitivity or dynamical sensitivity of the two walks, they may (and do) have very different properties.

We now define carefully the objects of interest. Let $X_1,X_2,\ldots$ be independent random variables satisfying
\[\P(X_i = 1) = \P(X_i = -1) = 1/2\]
for each $i\in\N$. Define, for each $n\ge 0$,
\[Y_n = \sum_{j=1}^n X_j\]
and
\[Z_n = \sum_{k=1}^n \prod_{j=1}^k X_j\]
where we take the empty sum to be zero, so $Y_0=Z_0=0$. We call $Y = (Y_n,\, n\ge 0)$ the compass random walk, and $Z = (Z_n,\, n\ge 0)$ the switch random walk. We can think of $Y = Y(X)$ and $Z=Z(X)$ as functions of the sequence of random variables $X=(X_1,X_2,\ldots)$. It is easy to see that, although they are different functions, the two walks $Y$ and $Z$ have the same distribution.

We now introduce dynamical versions of these random walks. For each $j\ge 1$, let $(N_j(t), t\ge0)$ be an independent Poisson process of rate $1$, and for each $i\ge 0$, let $X_j^i$ be an independent random variable with $\P(X_j^i = 1) = \P(X_j^i = -1) = 1/2$. Then define
\[X_j(t) = X_j^i \,\,\text{ whenever }\,\, N_j(t) = i.\]
In words, $X_j(t)$ has the same distribution as $X_j$ and rerandomises itself at the times of the Poisson process $N_j(t)$. Write $Y(t) = Y(X(t))$ and $Z(t) = Z(X(t))$, or more explicitly
\[Y_n(t) = \sum_{j=1}^n X_j(t) \,\,\,\,\,\,\text{ and }\,\,\,\,\,\, Z_n(t) = \sum_{k=1}^n \prod_{j=1}^k X_j(t)\]
for each $n\ge 0$.

For each fixed $t\ge 0$, the sequences $Y(t) = (Y_0(t), Y_1(t),\ldots)$ and $Z(t) = (Z_0(t), Z_1(t),\ldots)$ are simple symmetric random walks and therefore recurrent, in that $Y_n(t)=0$ for infinitely many values of $n$ almost surely, and similarly for $Z_n(t)$. Benjamini, H\"aggstr\"om, Peres and Steif \cite[Corollary 1.10]{benjamini_et_al:dynamic_sensitivity} showed that recurrence for $Y$ is \emph{dynamically stable} in that
\[\P(\forall t\ge0,\,\, Y_n(t) = 0 \text{ for infinitely many values of } n) = 1.\]
Our main result is that, in contrast, recurrence for $Z$ is dynamically sensitive. Define
\[\Ecal = \{t\in [0,1] : Z_n(t) \to \infty \text{ as } n\to\infty\},\]
\[\Ecal_0 = \{t\in[0,1] : \liminf_{n\to\infty} Z_n(t) > 0\},\]
and more generally for $\alpha\ge 0$,
\[\Ecal_\alpha = \Big\{t\in [0,1] : \liminf_{n\to\infty}\frac{Z_n(t)}{n^\alpha} > 0\Big\}.\]

\begin{thm}\label{mainthm}
There exist exceptional times of transience for the switch random walk: $\Ecal$ is non-empty almost surely. In fact, the Hausdorff dimension of $\Ecal_\alpha$ equals $1/2$ almost surely for any $\alpha\in[0,1/2)$. On the other hand, $\Ecal_\alpha$ is empty almost surely for any $\alpha>1/2$.
\end{thm}

It is an interesting question as to whether $\Ecal_{1/2}$ is empty or not. It is possible that the methods that we use to prove Theorem \ref{mainthm} could be extended to investigate this more delicate case, but this would require more detailed analysis of random walk sample paths that is beyond the scope of this paper.

We also show that the event that $Z_n$ is positive is noise sensitive. In fact we prove a stronger quantitative noise sensitivity result.

\begin{thm}\label{NSthm}
Let $(\eps_n, n\ge 1)$ be any sequence in $(0,1)$ such that $n\eps_n\to\infty$. The sequence of events $(\{Z_n>0\}, n\ge 1)$ is quantitatively noise sensitive with respect to the sequence $(\eps_n, n\ge 1)$, by which we mean that
\[\P(Z_n(0)>0 \text{ and } Z_n(\eps_n)>0) - \P(Z_n(0)>0)^2 \to 0\]
as $n\to\infty$.
\end{thm}
We note that the usual definition of (quantitative) noise sensitivity uses $-\log(1-\eps_n)$ in place of $\eps_n$ above, but since $\eps_n\in(0,1)$, this is equivalent to our statement.

We observe that if $\liminf n\eps_n <\infty$, then for arbitrarily large values of $n$ none of the first $n$ bits are rerandomised by time $\eps_n$, and therefore one cannot expect the events $\{Z_n(0)>0\}$ and $\{Z_n(\eps_n)>0\}$ to decorrelate. In this sense Theorem \ref{NSthm} is as strong as it possibly could be; we say that the events $(\{Z_n>0\}, n\ge1)$ are \emph{maximally noise sensitive}.

Again, Theorem \ref{NSthm} is in stark contrast to the corresponding statement for the compass random walk. In fact, the event that $Y_n$ is positive is known to be noise \emph{stable} \cite{benjamini_et_al:noise_sens}, in that
\[\lim_{\eps\to 0} \sup_n \P(\text{sign}\,Y_n(0) \neq \text{sign}\,Y_n(\eps)) = 0.\]

\section{Background and notation}

\subsection{Existing literature}

Noise sensitivity and dynamical sensitivity has been an active area of research in probability since the papers of H{\"a}ggstr{\"o}m, Peres and Steif \cite{haggstrom_peres_steif:dynamical_perco} and Benjamini, Kalai and Schramm \cite{benjamini_et_al:noise_sens}. One of the highlights of the subject is the proof that the existence of an infinite component in critical percolation in two dimensions is dynamically sensitive \cite{garban_pete_schramm:fourier, schramm_steif:noise_sens_percolation}. The survey of Steif \cite{steif:survey} and book by Garban and Steif \cite{garban_steif:NSbook} provide further background and references.

Benjamini, H\"aggstr\"om, Peres and Steif \cite{benjamini_et_al:dynamic_sensitivity} considered many properties of a quite general dynamical sequence of random variables, incorporating results on what we call the compass random walk $Y$. In particular they showed that for the compass random walk, the strong law of large numbers and the law of the iterated logarithm are both dynamically stable: almost surely there are no exceptional times at which either of these laws does not hold for $Y(t)$. It is not too difficult to check that the strong law of large numbers is also dynamically stable for the switch random walk, but it follows from our results that the law of the iterated logarithm is dynamically sensitive; indeed, Theorem \ref{mainthm} implies that there almost surely exist times $t$ at which $Z_n(t)$ is negative for all large $n$.

Benjamini et al \cite{benjamini_et_al:dynamic_sensitivity} also considered random walks in higher dimensions. They showed that in $\Z^d$, transience for the compass random walk (or rather its obvious analogue) is dynamically stable when $d\ge 5$. For $d\in\{3,4\}$ they showed that transience is dynamically sensitive and the set of exceptional times almost surely has Hausdorff dimension $(4-d)/2$. They conjectured that for $d=2$ recurrence should be dynamically sensitive, which was proven by Hoffman \cite{hoffman:recurrence_dynamically_sensitive}, who also showed that the Hausdorff dimension of the set of exceptional times of transience is $1$ almost surely. Further properties of dynamical random walks were investigated by Khoshnevisan, Levin and M\'endez-Hern\'andez \cite{khoshnevisan_et_al:dynamical_gaussian_RWs, khoshnevisan:exceptional_times_dynamical_RWs}.

The sequences $\{Y_n>0\}$ and $\{Z_n>0\}$ have exactly the same distribution---as sequences---and yet one is noise stable and one is noise sensitive. Warren \cite{warren:tanaka}, inspired by work of Tsirelson \cite{tsirelson:triple_points}, gave a similar example of such a pair: writing
\[W_n = \sum_{k=1}^n\text{sign}(W_{k-1})X_k,\]
the process $(W_n, n\ge 0)$ is also a simple symmetric random walk, and therefore has the same distribution as $(Y_n, n\ge 0)$, yet the events $\{W_n>0\}$ are noise sensitive.

The object that we refer to as the switch random walk is also known by other names. It has been called the \emph{coin-turning} random walk by Engl\"ander and Volkov who introduced more general (static) versions in \cite{englander_volkov:turning_coin}, and these were further studied by Engl\"ander, Volkov and Wang \cite{englander_volkov_wang:coin_turning_scaling}. It has also been called the \emph{bootstrap} random walk by Collevecchio, Hamza and Shi, who studied the pair $(Y,Z)$ in \cite{collevecchio_hamza_shi:bootstrap_rws}; Collevecchio, Hamza and Liu gave a further generalisation in \cite{collevecchio_hamza_liu:invariance_princ_bootstrap_rw}.

\subsection{Layout of article}

This article is organised as follows. In Section \ref{sketch_sec} we give a rough sketch of the proofs of Theorems \ref{mainthm} and \ref{NSthm}. We then carry out the proof of Theorem \ref{NSthm} in Section \ref{NSsec}. The proof of Theorem \ref{mainthm} is substatially more complex, and we give an outline in Section \ref{outlineproof}, which reduces the bulk of the task to proving two propositions, Proposition \ref{EPhifinite} for the lower bound on the Hausdorff dimension and Proposition \ref{influenceprop} for the upper bound, together with several technical lemmas. The proof of Proposition \ref{EPhifinite} is the most interesting part of the article and substantially different from existing proofs of related results. Rather than relying on the methods detailed in \cite{garban_steif:NSbook} such as randomised algorithms or the spectral sample, it instead uses more hands-on methods, leaning heavily on the independence of increments of random walks. We carry this out in Section \ref{EPhisec}. Then in Section \ref{influencesec} we prove Proposition \ref{influenceprop}, which mainly consists of elementary but intricate approximations. Finally, in Section \ref{techlemsec} we prove the technical lemmas required to complete the proof of Theorem \ref{mainthm}.

\subsection{Notation and preparatory results}

Throughout, we write $f(n)\lesssim g(n)$ if there exists a constant $c\in(0,\infty)$ such that $f(n)\le c g(n)$ for all large $n$, and $f(n)\asymp g(n)$ if both $f(n)\lesssim g(n)$ and $g(n)\lesssim f(n)$. We use $\approx$ only in heuristics to mean ``is roughly equal to''. We write $\P_x$ for the probability measure under which our random walks begin from $x$, rather than $0$. To be precise, we mean that under $\P_x$,
\[Z_n = x + \sum_{k=1}^n \prod_{j=1}^k X_j\]
and similarly for $Z_n(t)$, $Y_n$ and $Y_n(t)$.

We will use the Fortuin-Kasteleyn-Ginibre (FKG) inequality \cite{FKG} using the partial order on $\{-1,1\}^\mathbb{N}$ given by setting $(x_1,x_2,\ldots)\le (y_1,y_2,\ldots)$ if $x_i\le y_i$ for all $i\in\mathbb{N}$. This says that if $f$ and $g$ are either both increasing functions or both decreasing functions with respect to this partial order, then
\begin{equation}\label{fkg1}
\E[f(X)g(X)]\ge \E[f(X)]\E[g(X)]
\end{equation}
and if $f$ is increasing but $g$ is decreasing, then
\begin{equation}\label{fkg2}
\E[f(X)g(X)]\le \E[f(X)]\E[g(X)].
\end{equation}

We gather here some useful and well-known facts about simple symmetric random walks.

\begin{lem}\label{LCLT}
Suppose that $j\ge 2$. If $|z|\le j^{3/4}$ and $z\equiv j$ (mod 2), then
\[\P(Z_j = z) \asymp \frac{1}{j^{1/2}} \exp\Big(-\frac{z^2}{2j}\Big).\]
If $z\not\equiv j$ (mod 2) then $\P(Z_j=z)=0$.
\end{lem}

\begin{proof}
This is simply a version of the local central limit theorem: see for example \cite[Proposition 2.5.3 and Corollary 2.5.4]{lawler_limic:random_walk}.
\end{proof}

\begin{lem}\label{chernoff}
For any $j\ge 2$ and $x>0$,
\[\P(Z_j \ge x) \le \exp\Big(-\frac{x^2}{2j}\Big).\]
\end{lem}

\begin{proof}
This is an application of a simple Chernoff-style bound. For any $\lambda>0$,
\[\P(Z_j \ge x) \le \E[e^{\lambda Z_j}]e^{-\lambda x} = \E[e^{\lambda X_1}]^j e^{-\lambda x} = \Big(\frac{e^\lambda + e^{-\lambda}}{2}\Big)^j e^{-\lambda x}.\]
Noting that 
\[\frac{e^\lambda + e^{-\lambda}}{2} = \sum_{i=0}^\infty \frac{\lambda^{2i}}{(2i)!} \le \sum_{i=0}^\infty \frac{(\lambda^2/2)^i}{i!} = e^{\lambda^2/2},\]
we get
\[\P(Z_j \ge x) \le \exp\Big(\frac{\lambda^2 j}{2} - \lambda x\Big)\]
and choosing $\lambda = x/j$ gives the result.
\end{proof}

\begin{lem}\label{refl}
For any $z,j\in\mathbb{N}$,
\[\P(Z_i > -z \,\,\,\,\forall i=1,\ldots,j) = \P(Z_j \in [-z+1,z]).\]
\end{lem}

\begin{proof}
This is a version of the reflection principle. Note that
\begin{align*}
\P(Z_i > -z \,\,\,\,\forall i=1,\ldots,j) &= \P(Z_i > -z \,\,\,\,\forall i=1,\ldots,j, \,\, Z_j\ge -z+1)\\
&= \P(Z_j\ge -z+1) - \P(\exists i\le j : Z_i \le -z, \,\, Z_j \ge -z+1).
\end{align*}
Now by reflecting the random walk at the first hitting time of $-z$ (applying the strong Markov property), we have
\[\P(\exists i\le j : Z_i \le -z, \,\, Z_j \ge -z+1) = \P(Z_j \le -z-1) = \P(Z_j \ge z+1),\]
which establishes the result.
\end{proof}

\begin{cor}\label{reflcor}
For any $n\ge 1$,
\[\P(Z_i > 0 \,\,\,\,\forall i=1,\ldots,n) \asymp n^{-1/2}.\]
\end{cor}

\begin{proof}
We have
\[\P(Z_i>0\,\,\,\,\forall i = 1,\ldots,n) = \P(Z_1 = 1, \, Z_i>0\,\,\,\,\forall i = 2,\ldots,n) = \frac12 \P_1(Z_i>0\,\,\,\,\forall i=1,\ldots, n-1).\]
Applying Lemma \ref{refl}, the above equals $\frac12\P_1(Z_{n-1}\in[0,1])$, and by Lemma \ref{LCLT} this is of order $n^{-1/2}$.
\end{proof}

\section{Sketch proofs}\label{sketch_sec}

For $t\ge 0$ let $I_0(t) = 0$, and for $k\ge 1$ define
\[I_k(t) = \min\{i > I_{k-1}(t) : X_i(t)\neq X_i(0)\}.\]
We think of $t$ being small, so that for many indices $i$ we have $X_i(t) = X_i(0)$, and we call $I_k(t)$ the ``$k$th change'' (at time $t$ relative to time $0$). We call the steps of the random walk between $0=I_0(t)$ and $I_1(t)$ the \emph{first period}, the steps between $I_1(t)$ and $I_2(t)$ the \emph{second period}, and so on. For each $k$ we let $J_k(t) = I_k(t)-I_{k-1}(t)$ be the length of the $k$th period.

Our first key observation is that the increments of $Z_n(0)$ and $Z_n(t)$ are equal during odd periods (that is, for $n\in[I_{2k},I_{2k+1}(t)-1]$); and the increments of $Z_n(0)$ and $-Z_n(t)$ are equal during even periods (that is, for $n\in[I_{2k+1}(t),I_{2k+2}(t)-1]$). See Figure \ref{reflectionfig}.

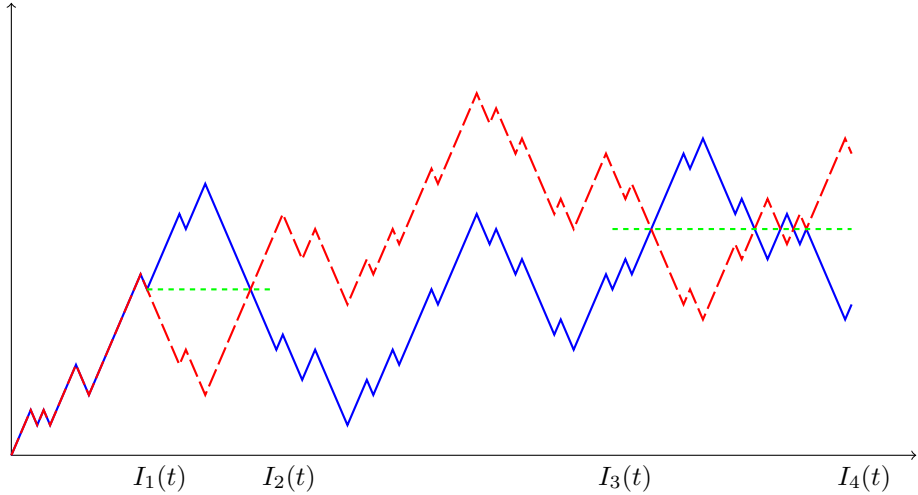
\begin{figure}[h!]
\centering
\begin{tikzpicture}[yscale=0.2, xscale=0.085]
\draw [<->] (0,30) -- (0,0) -- (140,0);
\draw[thick, blue] (0,0) -- (3,3) -- (4,2) -- (5,3) -- (6,2) -- (10,6) -- (12,4) -- (20,12) --  (21,11) -- (26, 16) -- (27,15) -- (30,18) -- (35,13) -- (41,7) -- (42,8) -- (45,5) -- (47,7) -- (52,2) -- (55,5) -- (56,4) -- (59,7) -- (60,6) -- (65,11) -- (66,10) -- (72,16) -- (74,14) -- (75,15) -- (78,12) -- (79,13) -- (84,8) -- (85,9) -- (87, 7) -- (92, 12) -- (93,11) -- (95,13) -- (96,12) -- (104,20) -- (105,19) -- (107,21) -- (112,16) -- (113,17) -- (117,13) -- (120,16) -- (122,14) -- (123,15) -- (129,9) -- (130,10);
\draw[thick, red, dash pattern={on 7pt off 2pt}] (0,0) -- (3,3) -- (4,2) -- (5,3) -- (6,2) -- (10,6) -- (12,4) -- (20,12) --  (21,11) -- (26,6) -- (27,7) -- (30,4) -- (35,9) -- (41,15) -- (42,16) -- (45,13) -- (47,15) -- (52,10) -- (55,13) -- (56,12) -- (59,15) -- (60,14) -- (65,19) -- (66,18) -- (72,24) -- (74,22) -- (75,23) -- (78,20) -- (79,21) -- (84,16) -- (85,17) -- (87, 15) -- (92, 20) -- (93,19) -- (95,17) -- (96,18) -- (104,10) -- (105,11) -- (107,9) -- (112,14) -- (113,13) -- (117,17) -- (120,14) -- (122,16) -- (123,15) -- (129,21) -- (130,20);
\node[below] at (23,0) {$I_1(t)$};
\node[below] at (43,0) {$I_2(t)$};
\node[below] at (95,0) {$I_3(t)$};
\node[below] at (132,0) {$I_4(t)$};
\draw[green, thick, dash pattern={on 2pt off 2pt}] (21,11) -- (40,11); 
\draw[green, thick, dash pattern={on 2pt off 2pt}] (93,15) -- (130,15);
\end{tikzpicture}
\caption{A realisation of $Z(0)$ in blue and $Z(t)$ in red (dashed) for the first four periods. The dotted green lines mark the lines of reflection.}
\label{reflectionfig}
\end{figure}

To see why Theorem \ref{NSthm} is true, let $t=\eps\in(0,1)$ and run the random walks up to step $n$. Let $U_n(t)$ be the sum of the increments of $Z_n(0)$ over odd periods up to step $n$, and $V_n(t)$ be the sum of the increments over even periods up to step $n$. Then clearly
\[Z_n(0) = U_n(t) + V_n(t).\]
(Note that $U_n(t)$ and $V_n(t)$ depend on $t$ because the periods depend on $t$, even though $Z_n(0)$ itself does not depend on $t$.) Of course, we can also write $Z_n(t)$ as the sum of its increments over odd periods, plus the sum of its increments over even periods. But the increments of $Z_n(t)$ over odd periods are equal to the increments of $Z_n(0)$ over odd periods, and the increments of $Z_n(t)$ over even periods are precisely \emph{minus} the increments of $Z_n(0)$ over even periods. Thus
\[Z_n(t) = U_n(t) - V_n(t).\]
As a result,
\[\P(Z_n(0)>0 \text{ and } Z_n(t)>0) = \P(U_n(t) + V_n(t)>0 \text{ and } U_n(t) - V_n(t)>0) = \P(U_n(t)>|V_n(t)|).\]
Now we note that---as long as $t\gg 1/n$, so that there are many periods by step $n$---the quantities $U_n(t)$ and $V_n(t)$ have \emph{almost} the same distribution when $n$ is large, and are \emph{almost} independent. They are also symmetric and have small probability of being equal or equalling zero. If $U$ and $V$ are independent symmetric continuous random variables, then $\P(U>|V|)=1/4$. Approximating this statement with $U_n(t)$ and $V_n(t)$ in place of $U$ and $V$ gives that
\[\P(Z_n(0)>0 \text{ and } Z_n(t)>0) \to 1/4\]
as $n\to\infty$, which is what is needed to prove Theorem \ref{NSthm} since clearly $\P(Z_n(0)>0)^2\to 1/4$.

Theorem \ref{mainthm} is significantly more difficult to prove. We give a sketch of a proof of the existence of exceptional times, whose main ideas are also the key to the most difficult part of calculating the Hausdorff dimension of the set of such times. There will be a much more detailed proof outline in Section \ref{outlineproof}.

It is simpler to deal with $\Ecal_0$ rather than $\Ecal$ or $\Ecal_\alpha$ for much of the proof. We define the event
\[P_n(t) = \{Z_k(t)>0 \,\,\,\,\forall k\in\{1,\ldots,n\}\},\]
and consider
\[\kappa_n = \int_0^1 \ind_{P_n(t)} \d t,\]
the Lebesgue amount of time in $[0,1]$ that the random walk $Z(t)$ stays positive for its first $n$ steps. To show the existence of exceptional times, ignoring some technical issues, it essentially suffices to show that
\[\E[\kappa_n^2] \le C\E[\kappa_n]^2\]
for some finite constant $C$, from which we can deduce that $\P(\kappa_n > 0) \ge 1/C$ and let $n\to\infty$.

For the first moment, by Fubini's theorem and stationarity,
\[\E[\kappa_n] = \int_0^1 \P(P_n(t)) \d t = \int_0^1 \P(P_n(0)) \d t = \P(P_n(0)).\]
Corollary \ref{reflcor} tells us that $\P(P_n(0))\asymp n^{-1/2}$.

For the second moment, again applying Fubini's theorem and stationarity, a simple argument (using Fubini's theorem and stationarity, and which we will give in full later) gives
\[\E[\kappa_n^2] \le 2\int_0^1 \P(P_n(0) \cap P_n(t)) \d t.\]
Our task is therefore to show that $\int_0^1 \P(P_n(0)\cap P_n(t)) \lesssim \P(P_n(0))^2 \asymp n^{-1}$.

During the even periods, the increments of $Z(0)$ and $Z(t)$ are mirrored. One can use this to show that the probability that both $Z(0)$ and $Z(t)$ remain positive over an even period is smaller than the square of the probability that $Z(0)$ stays positive over the same period. The total length of the even periods is roughly $n/2$ provided $t$ is not too small, and so (skipping over several important details) we might hope that, at least when $t$ is not too small,
\[\P(P_n(0)\cap P_n(t)) \lesssim \P(Z_{n/2}(0)>0)^2.\]
The details required to show this involve sewing together the increments over the even periods to create one random walk path of length roughly $n/2$. It is possible to do this in a very simple and natural way, except for one remaining issue: we cannot ignore the first period, on which the two random walks $Z(0)$ and $Z(t)$ are equal. On this period clearly the best upper bound we can get on the probability that both random walks stay positive is simply $\P(Z_{I_1(t)-1}(0)>0)$, rather than this quantity squared. A more reasonable overall upper bound is therefore
\[\P(P_n(0)\cap P_n(t)) \lesssim \frac{\P(Z_{n/2}(0)>0)^2}{\P(Z_{I_1(t)-1}(0)>0)}.\]
This does indeed hold, and since $I_1(t)\approx 2/t$, we have $\P(Z_{I_1(t)-1}(0)>0)\asymp (2/t)^{-1/2}$, so that
\[\int_0^1 \P(P_n(0)\cap P_n(t)) \d t \lesssim \int_0^1 \frac{n^{-1}}{t^{1/2}} \d t \asymp n^{-1}\]
as required. One may further note that an extra factor of $t^{-\gamma}$ in the integral would not make any difference to the calculation provided that $\gamma<1/2$, which combined with Frostman's lemma essentially gives us the lower bound of $1/2$ on the Hausdorff dimension.

\section{Proof of Theorem \ref{NSthm}: noise sensitivity for $\{Z_n>0\}$}\label{NSsec}

Fix a sequence $(\eps_n, n\ge1)$ with $\eps_n\in(0,1)$ for all $n$ and $n\eps_n\to\infty$. Many of the definitions in this section will depend implicitly on $\eps_n$. Recall that for $t\ge0$ we defined $I_0(t) = 0$, and for $k\ge 1$,
\[I_k(t) = \min\{i > I_{k-1}(t) : X_i(t)\neq X_i(0)\},\]
the start of the $(k+1)$th period. Let
\[K(n) = 2\lfloor n(1-e^{-\eps_n})/4 \rfloor.\]
We note that, since each $X_i$ has rerandomised by time $\eps_n$ with probability $1-e^{-\eps_n}$, the period length $I_k(\eps_n)-I_{k-1}(\eps_n)$ is a Geometric random variable of parameter $(1-e^{-\eps_n})/2$. Thus by the law of large numbers we have $I_{K(n)}(\eps_n) \approx n$.

There will be three main parts to this proof. In the first part, we show that the probability that the sum of the increments of a random walk on the odd periods is larger than the modulus of the sum of the increments on the even periods converges to $1/4$. In the second part, we will prove Theorem \ref{NSthm} but with $I_{K(n)}(\eps_n)$ in place of $n$. Finally, in the third part, we will transfer from using $I_{K(n)}(\eps_n)$ to $n$.

\vspace{2mm}

\noindent
\textbf{Part 1: Probability that sum of increments on odd periods exceed modulus of sum of increments on even periods converges to $1/4$.}\\
Define
\[U_n = \sum_{i=1}^{I_1(\eps_n) - 1} X_i + \sum_{i=I_2(\eps_n)}^{I_3(\eps_n) - 1} X_i + \ldots + \sum_{i=I_{K(n)-2}(\eps_n)}^{I_{K(n)-1}(\eps_n) -1}X_i + X_{I_{K(n)}(\eps_n)}\]
and
\[V_n = \sum_{i=I_1(\eps_n)}^{I_2(\eps_n) - 1} X_i + \sum_{i=I_3(\eps_n)}^{I_4(\eps_n) - 1} X_i + \ldots + \sum_{i=I_{K(n)-1}(\eps_n)}^{I_{K(n)}(\eps_n) -1}X_i.\]
In words, $U_n$ is the sum of the increments of a simple symmetric random walk (in fact $Y$, though this is not important) over the odd periods up to step roughly $n$, and $V_n$ is the sum over the even periods up to step roughly $n$. This is, of course, not quite true, since $I_{K(n)}(\eps_n)$ is unlikely to be exactly $n$. On the positive side, this gives $U_n$ and $V_n$ some nice properties: in particular, they are identically distributed.

We claim that
\[\lim_{n\to\infty} \P(U_n + V_n > 0 \text{ and } U_n - V_n > 0 ) = 1/4.\]
To see this, we observe that
\begin{align*}
1 &= \P(U_n>V_n>0) + \P(U_n>-V_n>0) + \P(V_n>U_n>0) + \P(-V_n>U_n>0)\\
&\hspace{4mm} + \P(U_n<V_n<0) + \P(U_n<-V_n<0) + \P(V_n<U_n<0) + \P(-V_n<U_n<0)\\
&\hspace{8mm} + \P(U_n = 0 \text{ or } V_n = 0 \text{ or } U_n=V_n \text{ or } U_n=-V_n).
\end{align*}
The first eight terms are all equal, and the last tends to $0$ as $n\to\infty$. Thus
\begin{align*}
\P(U_n + V_n > 0 \text{ and } U_n - V_n > 0 ) &= \P( U_n > |V_n| )\\
&= \P( U_n > V_n > 0 ) + \P( U_n > -V_n > 0 ) + \P(U_n>V_n=0)\\
&\to 1/8 + 1/8 + 0 = 1/4
\end{align*}
as claimed.

\vspace{2mm}

\noindent
\textbf{Part 2: Proving Theorem \ref{NSthm} but with $I_{K(n)}(\eps_n)$ in place of $n$.}\\
Noting that $K(n)$ is even, we now let
\[U'_n = Z_{I_1(\eps_n)-1}(0) + \sum_{\substack{k=3\\ k\text{ odd}}}^{K(n)-1} \big(Z_{I_k(\eps_n)-1}(0)-Z_{I_{k-1}(\eps_n)-1}(0)\big) + Z_{I_{K(n)}(\eps_n)}(0) - Z_{I_{K(n)}(\eps_n)-1}(0)\]
and
\[V'_n = \sum_{\substack{k=2\\ k\text{ even}}}^{K(n)} (Z_{I_k(\eps_n)-1}(0)-Z_{I_{k-1}(\eps_n)-1}(0)).\]
Clearly we have $Z_{I_{K(n)}(\eps_n)}(0) = U'_n + V'_n$. Moreover, since the increments of $Z(\eps_n)$ and $Z(0)$ are equal on odd periods and mirrored on even periods, we have
\[Z_{I_{K(n)}(\eps_n)}(\eps_n) = U'_n - V'_n.\]
Thirdly, note that (again recalling that $K(n)$ is even) $U'_n$ and $V'_n$ have the same joint distribution as $U_n$ and $V_n$. Thus we have
\begin{align*}
\P(Z_{I_{K(n)}(\eps_n)}(0) > 0 \text{ and } Z_{I_{K(n)}(\eps_n)}(\eps_n) > 0) &= \P( U'_n + V'_n > 0 \text{ and } U'_n - V'_n > 0 )\\
&= \P( U_n + V_n > 0 \text{ and } U_n - V_n > 0 )
\end{align*}
which we have just shown (in Part 1) converges to $1/4$ as $n\to\infty$. Thus
\[\P(Z_{I_{K(n)}(\eps_n)}(0) > 0 \text{ and } Z_{I_{K(n)}(\eps_n)}(\eps_n) > 0) - \P(Z_{I_{K(n)}(\eps_n)}(0) > 0)^2 \to \frac{1}{4} - \Big(\frac{1}{2}\Big)^2 = 0,\]
establishing the theorem with $I_{K(n)}(\eps_n)$ in place of $n$.

We remark here that so far, the proof works for any value of $\eps_n\in(0,1)$. However, if $\eps_n$ is too small, then the value of $K(n)$ is not large, which will cause problems in the following.

\vspace{2mm}

\noindent
\textbf{Part 3: Transferring from $I_{K(n)}(\eps_n)$ to $n$.}\\
We claim that
\begin{equation}\label{part3main}
\P( Z_n(0) > 0 \text{ and } Z_n(\eps_n)>0 ) = \P(Z_{I_{K(n)}(\eps_n)}(0) > 0 \text{ and } Z_{I_{K(n)}(\eps_n)}(\eps_n) > 0 ) + o(1).
\end{equation}
We will use the elementary bounds, for any events $A$, $B$, $A'$ and $B'$,
\[\P(A\cap B) \le \P(A'\cap B') + \P(A\setminus A') + \P(B\setminus B')\]
and
\[\P(A\cap B) \ge \P(A'\cap B') - \P(A'\setminus A) - \P(B'\setminus B).\]
For the upper bound, using the first fact above,
\begin{align*}
\P( Z_n(0) > 0 \text{ and } Z_n(\eps_n)>0 ) &\le \P(Z_{I_{K(n)}(\eps_n)}(0) > 0 \text{ and } Z_{I_{K(n)}(\eps_n)}(\eps_n) > 0 )\\
&\hspace{20mm} + \P(Z_n(0) > 0 \text{ but } Z_{I_{K(n)}(\eps_n)}(0) \le 0)\\
&\hspace{30mm} + \P(Z_n(\eps_n) > 0 \text{ but } Z_{I_{K(n)}(\eps_n)}(\eps_n) \le 0),
\end{align*}
and for the lower bound, using the second fact above,
\begin{align*}
\P( Z_n(0) > 0 \text{ and } Z_n(\eps_n)>0 ) &\ge \P(Z_{I_{K(n)}(\eps_n)}(0) > 0 \text{ and } Z_{I_{K(n)}(\eps_n)}(\eps_n) > 0 )\\
&\hspace{20mm} - \P(Z_n(0) \le 0 \text{ but } Z_{I_{K(n)}(\eps_n)}(0) > 0)\\
&\hspace{30mm} - \P(Z_n(\eps_n) \le 0 \text{ but } Z_{I_{K(n)}(\eps_n)}(\eps_n) > 0).
\end{align*}
We will show that
\[\P(Z_n(0) > 0 \text{ but } Z_{I_{K(n)}(\eps_n)}(0) \le 0)\to 0;\]
the three other similar terms can be dealt with similarly. To do this, we first note that for any $x_n,y_n>0$,
\begin{multline}\label{splitKineq}
\P\big(Z_n(0) > 0 \text{ but } Z_{I_{K(n)}(\eps_n)}(0) \le 0\big) \le \P\big( |I_{K(n)}(\eps_n) - n| > x_n \big) + \P\big( Z_n(0) \in (0,y_n) \big)\\
+ \P\Big( Z_n(0) \ge y_n \text{ but } \min_{j\in[n-x_n,n+x_n]} Z_j(0) \le 0 \Big).
\end{multline}
We first consider $\P( |I_{K(n)}(\eps_n) - n| > x_n)$. We use Markov's inequality to see that
\[\P\big( |I_{K(n)}(\eps_n) - n| > x_n \big) \le \frac{\E\big[|I_{K(n)}(\eps_n) - n|^2\big]}{x_n^2},\]
and using the fact that $I_{K(n)}(\eps_n)$ is a sum of $K(n)$ independent Geometric random variables of parameter $(1-e^{-\eps_n})/2$, we have
\begin{align*}
\E\big[|I_{K(n)}(\eps_n) - n|^2\big] &= \text{Var}(I_{K(n)}(\eps_n)) + \E[I_{K(n)}(\eps_n)]^2 - 2n\E[I_{K(n)}(\eps_n)] + n^2\\
&= \frac{2K(n)(1+e^{-\eps_n})}{(1-e^{-\eps_n})^2} + \frac{4K(n)^2}{(1-e^{-\eps_n})^2} - \frac{4nK(n)}{1-e^{-\eps_n}} + n^2.
\end{align*}
Recalling that $K(n) = 2\lfloor n(1-e^{-\eps_n})/4 \rfloor$, the above is at most
\[\frac{n(1+e^{-\eps_n})}{1-e^{-\eps_n}} + n^2 - \Big(\frac{8n}{1-e^{-\eps_n}}\Big)\Big(\frac{n(1-e^{-\eps_n})}{4}-1\Big) + n^2 \le \frac{10n}{1-e^{-\eps_n}}.\]
Thus
\[\P\big( |I_{K(n)}(\eps_n) - n| > x_n \big) \le \frac{10n}{x_n^2(1-e^{-\eps_n})}.\]
Choosing the value $x_n = n^{5/8}/(1-e^{-\eps_n})^{3/8}$, we have
\begin{equation}\label{splitKterm1}
\P\big( |I_{K(n)}(\eps_n) - n| > x_n \big) \le \frac{10}{n^{1/4}(1-e^{-\eps_n})^{1/4}} \to 0
\end{equation}
by our assumption that $n\eps_n\to\infty$.

We now move on to the second term on the right-hand side of \eqref{splitKineq}. Choosing $y_n = n^{3/8}/\eps_n^{1/8}$, since $(Z_j(0), j\ge 0)$ is a simple symmetric random walk and $y_n\ll n^{1/2}$, by the central limit theorem we have
\begin{equation}\label{splitKterm2}
\P\big( Z_n(0) \in (0,y_n) \big) \to 0.
\end{equation}
For the final term in \eqref{splitKineq}, by the strong Markov property and Lemma \ref{refl},
\begin{align*}
\P\Big( Z_n(0) \ge y_n \text{ but } \min_{j\in[n-x_n,n+x_n]} Z_j(0) \le 0 \Big) &\le \P_0\Big(\max_{j\in[0,x_n]} Z_j(0) \ge y_n\Big) + \P_{y_n}\Big(\min_{j\in[0,x_n]} Z_j(0) \le 0\Big)\\
&= 2\big(1-\P(Z_{\lfloor x_n\rfloor}(0)\in[-y_n+1,y_n])\big).
\end{align*}
Since $x_n = n^{5/8}/(1-e^{-\eps_n})^{3/8} \ll n^{6/8}/\eps_n^{2/8} = y_n^2$, the central limit theorem tells us that the above also converges to zero as $n\to\infty$. Combining this with \eqref{splitKterm1} and \eqref{splitKterm2}, we see from \eqref{splitKineq} that
\[\P\big(Z_n(0) > 0 \text{ but } Z_{I_{K(n)}(\eps_n)}(0) \le 0\big)\to 0.\]
This, together with very similar bounds on the other three terms mentioned above, establishes \eqref{part3main}. In Part 2 we showed that
\[\lim_{n\to\infty}\P(Z_{I_{K(n)}(\eps_n)}(0) > 0 \text{ and } Z_{I_{K(n)}(\eps_n)}(\eps_n) > 0 ) = 1/4,\]
and clearly $\P(Z_n(0)>0)\to 1/2$, so the proof of Theorem \ref{NSthm} is complete.

\section{Outline of the proof of Theorem \ref{mainthm}: Hausdorff dimension of exceptional times is $1/2$}\label{outlineproof}

We now outline the main steps in turning the heuristic in Section \ref{sketch_sec} into a rigorous proof that the Hausdorff dimension of
\[\Ecal_\alpha = \Big\{t\in [0,1] : \liminf_{n\to\infty} \frac{Z_n(t)}{n^\alpha}>0\Big\}\]
is $1/2$ almost surely for any $\alpha\in[0,1/2)$. Since $\Ecal_\alpha\subset\Ecal_0$ for any $\alpha\ge0$, it suffices to give an upper bound on the dimension of $\Ecal_0$ and a lower bound on the dimension of $\Ecal_\alpha$ for $\alpha\in(0,1/2)$. This also, of course, implies that $\Ecal$ is non-empty almost surely and therefore that there exist exceptional times of transience. We will proceed by stating a series of results, whose proofs we delay until later sections.

\subsection{Lower bound on Hausdorff dimension of $\Ecal_\alpha$}\label{subsec:LBHD}

As in the sketch proof, we define the event
\[P_n(t) = \{Z_i(t)>0 \,\,\,\,\forall i=1,\ldots,n\},\]
and similarly
\[P_n = \{Z_i>0 \,\,\,\,\forall i=1,\ldots,n\}.\]
We will use these events for much of the proof. However, to consider $\Ecal_\alpha$ for $\alpha>0$, we will also need the more complicated events
\[P^\alpha_n(t) = \big\{Z_i(t)\ge i^\alpha \,\,\,\,\forall i=1,\ldots,n\big\}\]
and similarly for $P^\alpha_n$, defined for any $\alpha\ge 0$, though we will mostly think of $\alpha\in[0,1/2)$. Note that $P^0_n(t)=P_n(t)$.

Let
\[T^\alpha_n = \{t\in[0,1] : P^\alpha_n(t) \text{ holds}\}.\]
We write $\bar T^\alpha_n$ for the closure of $T^\alpha_n$ and $T^\alpha=\bigcap_n T^\alpha_n$. Finally define, for $\gamma\in[0,1)$,
\[\Phi^\alpha_n(\gamma) = \frac{1}{\P(P^\alpha_n)^2}\int_0^1 \int_0^1 \frac{\ind_{P^\alpha_n(s)\cap P^\alpha_n(t)}}{|t-s|^\gamma} \d s \d t.\]

Our lower bound on the Hausdorff dimension of $\Ecal_\alpha$ will be based on the following corollary of \cite[Lemma 6.2]{schramm_steif:noise_sens_percolation}, which in turn is an application of Frostman's lemma.

\begin{lem}\label{Hdimlower}
Suppose that for some $\alpha\ge 0$ and $\gamma\in(0,1)$ we have
\[\sup_n \E[\Phi^\alpha_n(\gamma)]<\infty.\]
Then the Hausdorff dimension of $\bigcap_n \bar T^\alpha_n$ is at least $\gamma$ with strictly positive probability.
\end{lem}

Given Lemma \ref{Hdimlower}, which we will prove in Section \ref{techlemsec}, our main task in proving the lower bound becomes to show that $\E[\Phi^\alpha_n(\gamma)]$ is bounded above for each $\alpha,\gamma<1/2$. This will be the most difficult (and most novel) part of our proof, and will be carried out in Section \ref{EPhisec}.

\begin{prop}\label{EPhifinite}
For any $\alpha,\gamma\in[0,1/2)$,
\[\sup_n \E[\Phi^\alpha_n(\gamma)]<\infty.\]
\end{prop}

Combining Lemma \ref{Hdimlower} and Proposition \ref{EPhifinite} tells us that for any $\alpha,\gamma\in[0,1/2)$, the Hausdorff dimension of $\bigcap_n \bar T^\alpha_n$ is at least $\gamma$ with strictly positive probability. This is not quite what was promised in Theorem \ref{mainthm}, which in fact says that the Hausdorff dimension of $\Ecal_\alpha$ is $1/2$ almost surely for any $\alpha\in[0,1/2)$. Moving from $\bigcap_n \bar T^\alpha_n$ to $T^\alpha$ is a technicality that can be handled in basically the same way as \cite[Lemma 3.2]{haggstrom_peres_steif:dynamical_perco}; and of course $T^\alpha\subset \Ecal_\alpha$. Finally, showing that the Hausdorff dimension of $\Ecal_\alpha$ is at least $1/2$ almost surely, rather than with positive probability, follows from standard ergodicity arguments (of course this cannot hold for $T^\alpha$, since with positive probability $Z_2(t)=0$ for all $t\in[0,1]$). The following lemmas take care of these steps. We will prove them in Section \ref{techlemsec}.

\begin{lem}\label{Ttechnicality}
For any $\alpha\ge 0$, we have
\[\bigcap_{n=1}^\infty \bar T^\alpha_n = \bigcap_{n=1}^\infty T^\alpha_n\]
almost surely.
\end{lem}

\begin{lem}\label{ergodic}
For each $\alpha\ge0$, the Hausdorff dimension of $\Ecal_\alpha$ is a constant (possibly depending on $\alpha$) almost surely.
\end{lem}

\subsection{Upper bound on Hausdorff dimension of $\Ecal_0$}\label{subsec:UBHD}

The following definitions are more or less standard in the noise sensitivity literature. For a function $f:\{-1,1\}^\N\to\R$, we say that $m\in\N$ is \emph{pivotal} for $f$ if
\[f(X_1,\ldots,X_{m-1},X_m,X_{m+1},X_{m+2},\ldots) \neq f(X_1,\ldots,X_{m-1},-X_m,X_{m+1},X_{m+2},\ldots).\]
Of course this definition depends on the realisation of $X_1,X_2,\ldots$, although we note that it is independent of the value of $X_m\in\{-1,1\}$. For an event $E$, we say that $m$ is pivotal for $E$ if $m$ is pivotal for the indicator function of $E$. We define the \emph{influence} of the $m$th bit (on $E$) to be
\[\mathcal I_m(E) = \P(m \text{ is pivotal for } E)\]
and the \emph{total influence} of $E$ to be
\[\mathcal I(E) = \sum_{m=1}^\infty \mathcal I_m(E).\]

For technical reasons, we will need the following generalisations of $P_n$ and $T$. For $k\in 2\Z_+$, define the event
\[P_{k,n} = \{Z_k=0, \, Z_i > 0 \,\,\forall i=k+1,\ldots,k+n\}\]
and let
\[T'_k = \{t\in[0,1] : Z_k(t)=0,\, Z_i(t)>0 \,\,\,\,\forall i=k+1,k+2,\ldots\}.\]

Our next lemma is just a rephrasing of \cite[Theorem 8.1]{schramm_steif:noise_sens_percolation} into our setting, and gives us a condition for bounding the Hausdorff dimension of $T'_k$ in terms of the total influence of $P_{k,n}$.

\begin{lem}\label{Hdimupper}
The Hausdorff dimension of $T'_k$ is almost surely at most
\[\liminf_{n\to\infty} \Big(1-\frac{\log \P(P_{k,n})}{\log\mathcal I(P_{k,n})}\Big)^{-1}.\]
\end{lem}

\begin{proof}
This is almost exactly the second part of the statement of \cite[Theorem 8.1]{schramm_steif:noise_sens_percolation} translated into our notation. There is an extra condition that the events $P_{k,n}$ must depend only on finitely many random variables, but this is clearly satisfied since $P_{k,n}$ depends only on $X_1,\ldots,X_{n+k}$.
\end{proof}

To implement Lemma \ref{Hdimupper} we now need an upper bound on the influences of $P_n$.

\begin{prop}\label{influenceprop}
For any $m=1,2,\ldots,n$, we have
\[\mathcal I_m(P_n) \asymp \frac{n-m+1}{n^{3/2}}.\]
\end{prop}

This result will be proved in Section \ref{influencesec}. Combining Proposition \ref{influenceprop} with Lemma \ref{Hdimupper} will give us the upper bound of $1/2$ on the Hausdorff dimension of $T^0$ and hence $\Ecal$. We carry out the details in Section \ref{sec:completethm1proof}.

\subsection{$\Ecal_\alpha$ is empty for $\alpha>1/2$}\label{subsec:alphabig}

The final part of Theorem \ref{mainthm} says that $\Ecal_\alpha$ is empty almost surely when $\alpha>1/2$. The proof of this fact follows a fairly standard argument. For $\alpha,t\ge0$ and $n\in\N$ define the event $L^\alpha_n(t) = \{Z_n(t) \ge n^\alpha\}$, and for $k\in\N$ let $\mathcal L^\alpha_n(k) = \int_0^k \ind_{L^\alpha_n(t)} \d t$. Note that
\begin{equation}\label{trivub}
\P(\mathcal L^\alpha_n(1) > 0) \le \P(\mathcal L^\alpha_n(1) > 0)\frac{\E[\mathcal L^\alpha_n(2)]}{\E[\mathcal L^\alpha_n(2) \ind_{\{\mathcal L^\alpha_n(1) > 0\}}]} = \frac{\E[\mathcal L^\alpha_n(2)]}{\E[\mathcal L^\alpha_n(2) \,|\, \mathcal L^\alpha_n(1) > 0]}.
\end{equation}
By Fubini's theorem and stationarity,
\[\E[\mathcal L^\alpha_n(2)] = \int_0^2 \P(Z_n(t) \ge n^\alpha) \d t = 2\P(Z_n \ge n^\alpha).\]
By Markov's inequality, for any $\lambda>0$,
\[\P(Z_n \ge n^\alpha) = \P(\exp(\lambda Z_n) \ge \exp(\lambda n^\alpha)) \le \E[\exp(\lambda Z_n)]\exp(-\lambda n^{\alpha}).\]
Since $Z_n$ is a sum of $n$ independent and identically distributed random variables,
\[\E[\exp(\lambda Z_n)] = \E[\exp(\lambda Z_1)]^n = (e^\lambda/2 + e^{-\lambda}/2)^n.\]
When $\lambda$ is small we have $e^\lambda/2 + e^{-\lambda}/2 \le 1+3\lambda^2/4$, so fixing $\alpha\in(1/2,1)$ and choosing $\lambda = n^{\alpha-1}$, for large $n$ we have
\[\E[\exp(\lambda Z_n)] \le \Big(1+\frac{3}{4}\lambda^2\Big)^n = \Big(1 + \frac{3}{4}n^{2\alpha-2}\Big)^n \le \exp\Big(\frac{3}{4}n^{2\alpha-1}\Big).\]
Thus, again with $\alpha\in(1/2,1)$ and $\lambda = n^{\alpha-1}$, for large $n$,
\begin{equation}\label{trivub1}
\E[\mathcal L^\alpha_n(2)] = 2\P(Z_n \ge n^\alpha) \le 2\exp\Big(\frac{3}{4}n^{2\alpha-1}\Big)\exp(-n^{2\alpha-1}) = 2\exp(-n^{2\alpha-1}/4).
\end{equation}

On the other hand, letting $T = \inf\{t\ge 0 : Z_n(t) \ge n^\alpha\}$,  we have
\[\E[\mathcal L^\alpha_n(2) \,|\, \mathcal L^\alpha_n(1) > 0] \ge \E\Big[\int_T^{T+1} \ind_{L^\alpha_n(t)} \d t \,\Big|\, \mathcal L^\alpha_n(1) > 0\Big].\]
Let $T' = \inf\{t\ge T : \text{one of the first $n$ steps rerandomises}\}$. Then clearly, provided $T<\infty$,
\[\int_T^{T+1} \ind_{L^\alpha_n(t)} \d t \ge (T'-T)\wedge 1.\]
However, by the strong Markov property, $T'-T$ is exponentially distributed with parameter $n$. Thus
\[\E\Big[\int_T^{T+1} \ind_{L^\alpha_n(t)} \d t \,\Big|\, \F_T \Big] \ge \E[(T'-T)\wedge 1] = \int_0^1 s\cdot ne^{-ns} \d s \ge \int_0^{1/n} ns e^{-ns} \d s \ge \frac{1}{2en}.\]
Thus
\[\E[\mathcal L^\alpha_n(2) \,|\, \mathcal L^\alpha_n(1) > 0] \ge \frac{1}{2en}.\]
Combining this with \eqref{trivub} and \eqref{trivub1}, for any $\alpha\in(1/2,1)$ we have
\[\P(\mathcal L^\alpha_n(1) > 0) \le 2\exp(-n^{2\alpha-1}/4)\cdot 2en.\]
By the Borel-Cantelli lemma, for any $\alpha\in(1/2,1)$, the probability that for infinitely many $n$, there exists a time in $[0,1]$ such that $L^\alpha_n(t)$ occurs, is zero. Thus $\Ecal_\alpha$ is empty almost surely. Since $\Ecal_{\alpha'}\subset \Ecal_\alpha$ for any $\alpha'\ge \alpha$, we also deduce the same for $\alpha\ge 1$.

\subsection{Completing the proof of Theorem \ref{mainthm}}\label{sec:completethm1proof}

We now tie together the results from Sections \ref{subsec:LBHD}, \ref{subsec:UBHD} and \ref{subsec:alphabig} to complete the proof of Theorem \ref{mainthm}.

\begin{proof}[Proof of Theorem \ref{mainthm}]
We showed in Section \ref{subsec:alphabig} that $\Ecal_\alpha$ is empty almost surely for $\alpha>1/2$, so it remains to show that the Hausdorff dimension of $\Ecal_\alpha$ is $1/2$ for any $\alpha\in[0,1/2)$. As stated at the beginning of Section \ref{outlineproof}, it suffices to show that the Hausdorff dimension of $\Ecal_\alpha$ is at least $1/2$ for $\alpha>0$ and the Hausdorff dimension of $\Ecal_0$ is at most $1/2$.

By Lemma \ref{Hdimlower} and Proposition \ref{EPhifinite}, we know that for any $\alpha,\gamma\in[0,1/2)$, the Hausdorff dimension of $\bigcap_n \bar T^\alpha_n$ is at least $\gamma$ with strictly positive probability. By Lemma \ref{Ttechnicality}, the same holds for $T^\alpha$, and since $T^\alpha\subset\Ecal_\alpha$, the same holds for $\Ecal_\alpha$. Lemma \ref{ergodic} then tells us that the Hausdorff dimension of $\Ecal_\alpha$ must be at least $1/2$ almost surely.

Moving on to the upper bound, take $k\in 2\Z_+$ and $m\in\{k+1,k+2,\ldots,k+n\}$. If $Z_k\neq 0$ then $m$ cannot be pivotal for $P_{k,n}$, so
\[\mathcal I_m(P_{k,n}) = \P(Z_k=0, \,\, m \text{ is pivotal for } P_{k,n}) = \P(Z_k=0)\P(m \text{ is pivotal for } P_{k,n}\,|\,Z_k=0).\]
But by the Markov property,
\[\P(m \text{ is pivotal for } P_{k,n}\,|\,Z_k=0) = \P(m-k \text{ is pivotal for } P_n) = \mathcal I_{m-k}(P_n).\]
Thus
\[\mathcal I(P_{k,n}) = \sum_{m=1}^{k} \mathcal I_m(P_{k,n}) + \sum_{m=k+1}^{k+n} \mathcal I_m(P_{k,n}) \le k + \P(Z_k=0) \sum_{m=1}^n \mathcal I_m(P_n),\]
and so, applying Proposition \ref{influenceprop},
\begin{equation}\label{IPnk}
\mathcal I(P_{k,n}) \lesssim k+ \frac{\P(Z_k=0)}{n^{3/2}}\sum_{m=1}^n (n-m+1) \asymp k+\P(Z_k=0) n^{1/2}.
\end{equation}

By the Markov property
\[\P(P_{k,n}) = \P(Z_k=0)\P(Z_i>0 \,\,\forall i=k+1,k+2,\ldots,k+n \,|\,Z_k=0) = \P(Z_k=0)\P(P_n),\]
and by Corollary \ref{reflcor} we have $\P(P_n)\asymp n^{-1/2}$. Combining this with \eqref{IPnk}, we see that there exist constants $c,c'\in(0,\infty)$ such that
\[\frac{-\log \P(P_{k,n})}{\log \mathcal I(P_{k,n})} \ge \frac{\frac12 \log n - \log c - \log \P(Z_k=0)}{\frac12 \log n + \log c' + \log(\P(Z_k=0)+kn^{-1/2})},\]
which converges to $1$ as $n\to\infty$ for each fixed $k$. From Lemma \ref{Hdimupper} we obtain that the Hausdorff dimension of $T_k'$ is almost surely at most $(1+1)^{-1} = 1/2$.

Finally,
\[\Ecal_0 = \{t\in[0,1] : \liminf_{n\to\infty} Z_n(t)>0 \} = \bigcup_k T_k'\]
which as a countable union of sets of Hausdorff dimension at most $1/2$ almost surely, itself has Hausdorff dimension at most $1/2$ almost surely. This completes the proof.
\end{proof}

\section{Proof of Proposition \ref{EPhifinite}: bounding $\E[\Phi^\alpha_n(\gamma)]$ from above}\label{EPhisec}

First note that, by Fubini's theorem,
\begin{align*}
\E[ \Phi^\alpha_n(\gamma)] &= \frac{1}{\P(P_n^\alpha)^2}\E\Big[\int_0^1 \int_0^1 \frac{\ind_{P_n^\alpha(s)\cap P_n^\alpha(t)}}{|t-s|^\gamma} \d s\, \d t\Big]\\
&= \frac{1}{\P(P_n^\alpha)^2}\int_0^1 \int_0^1 \frac{\P(P_n^\alpha(s) \cap P_n^\alpha(t))}{|t-s|^\gamma} \d s \, \d t.
\end{align*}
By stationarity, this is bounded above by
\[\frac{2}{\P(P_n^\alpha)^2} \int_0^1 \frac{\P(P_n^\alpha(0) \cap P_n^\alpha(t))}{t^\gamma} \d t,\]
and since $P_n^\alpha(u) \subset P_n(u)$ for any $\alpha,u\ge0$, this is at most
\[\frac{2}{\P(P_n^\alpha)^2} \int_0^1 \frac{\P(P_n(0) \cap P_n(t))}{t^\gamma} \d t.\]

The following lemma says that the probability of $P_n^\alpha$ is of the same order as the probability as $P_n$. It is a simple application of \cite[Theorem 2]{ritter:growth_RW_cond_positive} and we will prove it later in this section.

\begin{lem}\label{Pnalphaasymp}
For any $\alpha<1/2$,
\[\P(P_n^\alpha)\asymp \frac{1}{\sqrt n}.\]
\end{lem}

We now want to bound $\P(P_n(0) \cap P_n(t))$. As suggested in the sketch proof in Section \ref{sketch_sec}, the main idea is that on even periods two mirrored random walks (representing the walk at time $0$ and time $t$) must both be larger than $0$. The difficulty is in handling the dependencies between periods, and for this we need some more definitions.

For each $j\ge 1$, define the event
\[A_j(t) = \{Z_i(0)>0 \text{ and } Z_i(t)>0 \,\,\,\, \forall i\in[I_{j-1}(t), I_j(t)-1]\},\]
which says that our dynamical random walk is positive throughout the $j$th period at both time $0$ and time $t$ (recall the terminology from Section \ref{sketch_sec}). For each $i\ge 0$, let
\[W_i(t) = \frac{Z_i(0)+Z_i(t)}{2}.\]
Note that, for each $t$, during odd periods the increments of $W_i(t)$ are equal to the increments of $Z_i(0)$; and during even periods, $W_i(t)$ is constant. (When we talk about increments we mean as $i$ changes, keeping $t$ fixed.)

When $j$ is odd, define the event
\[A'_j(t) = \{W_i(t)>0 \,\,\,\, \forall i\in[I_{j-1}(t),I_j(t)-1]\}.\]
Note that, since $W_i(t)$ is the average of $Z_i(0)$ and $Z_i(t)$, if both of these are positive, then so is $W_i(t)$. That is, if $j$ is odd, then $A_j(t) \subset A'_j(t)$.

Making the same comparison when $j$ is even would not be useful since $W$ is constant. Instead, when $j$ is even, let $B^{(j)}_i(t)$, $i\ge0$ be an independent simple random walk started from $W_{I_{j-1}(t)-1}(t)$ and define
\[A'_j(t) = \{B^{(j)}_i(t) \in (0,2W_{I_{j-1}(t)-1}(t)) \,\,\,\, \forall i\in[1,J_j(t)]\}.\]
Figure \ref{4walks} shows a realisation of $Z(0)$, $Z(t)$, $W(t)$, $B^{(2)}(t)$ and $B^{(4)}(t)$.

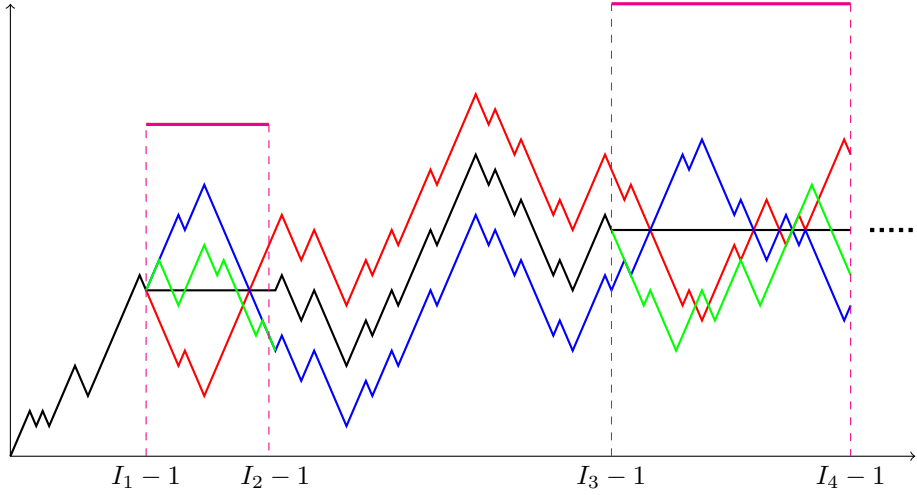
\begin{figure}
\centering
\begin{tikzpicture}[yscale=0.2, xscale=0.085]
\draw [<->] (0,30) -- (0,0) -- (140,0);
\draw [thick] (0,0) -- (3,3) -- (4,2) -- (5,3) -- (6,2) -- (10,6) -- (12,4) -- (20,12) -- (21,11) -- (21,11) -- (41,11) -- (42,12) -- (45,9) -- (47,11) -- (52,6) -- (55,9) -- (56,8) -- (59,11) -- (60,10) -- (65,15) -- (66,14) -- (72,20) -- (74,18) -- (75,19) -- (78,16) -- (79,17) -- (84,12) -- (85,13) -- (87, 11) -- (92, 16) -- (93,15) -- (130,15);
\draw[thick, red] (21,11) -- (26,6) -- (27,7) -- (30,4) -- (35,9) -- (41,15) -- (42,16) -- (45,13) -- (47,15) -- (52,10) -- (55,13) -- (56,12) -- (59,15) -- (60,14) -- (65,19) -- (66,18) -- (72,24) -- (74,22) -- (75,23) -- (78,20) -- (79,21) -- (84,16) -- (85,17) -- (87, 15) -- (92, 20) -- (93,19) -- (95,17) -- (96,18) -- (104,10) -- (105,11) -- (107,9) -- (112,14) -- (113,13) -- (117,17) -- (120,14) -- (122,16) -- (123,15) -- (129,21) -- (130,20);
\draw[thick, blue] (21,11) -- (26, 16) -- (27,15) -- (30,18) -- (35,13) -- (41,7) -- (42,8) -- (45,5) -- (47,7) -- (52,2) -- (55,5) -- (56,4) -- (59,7) -- (60,6) -- (65,11) -- (66,10) -- (72,16) -- (74,14) -- (75,15) -- (78,12) -- (79,13) -- (84,8) -- (85,9) -- (87, 7) -- (92, 12) -- (93,11) -- (95,13) -- (96,12) -- (104,20) -- (105,19) -- (107,21) -- (112,16) -- (113,17) -- (117,13) -- (120,16) -- (122,14) -- (123,15) -- (129,9) -- (130,10);
\draw [thick, green] (21,11) -- (23,13) -- (26,10) -- (30,14) -- (32,12) -- (33,13) -- (38,8) -- (39,9) -- (41,7); 
\draw [thick, green] (93,15) -- (98,10) -- (99,11) -- (103,7) -- (107,11) -- (109,9) -- (113,13) -- (116,10) -- (124,18) -- (130,12);
\draw[magenta, dashed] (21,0) -- (21,22); 
\draw[magenta, dashed] (40,22) -- (40,0);
\draw[very thick, magenta] (21,22) -- (40,22);
\draw[magenta, dashed] (93,0) -- (93,30);
\draw[magenta, dashed] (130,30) -- (130,0);
\draw[very thick, magenta] (93,30) -- (130,30);
\draw [dotted, ultra thick] (133,15) -- (140,15);
\node[below] at (21,0) {$I_1-1$};
\node[below] at (41,0) {$I_2-1$};
\node[below] at (93,0) {$I_3-1$};
\node[below] at (130,0) {$I_4-1$};
\end{tikzpicture}
\caption{A realisation of $Z(0)$ and $Z(t)$ (blue/red), $W(t)$ (black), $B^{(2)}(t)$ and $B^{(4)}(t)$ (both green) for the first four periods.\label{4walks}}
\end{figure}

We need to rule out some unlikely events. Let
\[\odd = \{J_3(t) + J_5(t) + \ldots + J_{2\lfloor nt/8\rfloor+1}(t) \ge n/8\},\]
\[\even = \{J_2(t) + J_4(t) + \ldots + J_{2\lfloor nt/8\rfloor}(t) \ge n/8\},\]
\[E_n(t) = \odd\cap\even \,\,\,\, \text{ and } \,\,\,\, E'_n(t) = \{I_{2\lfloor nt/8\rfloor+1}(t)\le n\}.\]
We note that for each $j$, when $t$ is small $J_j(t)$ has expectation roughly $2/t$, so when $n$ is large the above events should all occur with probability close to $1$. The following lemma, which we prove later in the section, quantifies this more precisely.

\begin{lem}\label{Eproblarge}
There exists a constant $\delta>0$ such that for any $t\in[0,1]$ and $n\in\N$,
\[\P(E_n(t)^c) + \P(E'_n(t)^c) \le \exp(-\delta nt).\]
\end{lem}

For now we will work on the event $E_n(t)$. Also define, for $k\in\N$,
\[V_k(t) = \bigcap_{j=1}^k A_j(t) \,\,\,\, \text{ and } V'_k(t) = \bigcap_{j=1}^k A'_j(t).\]

Our next result translates the probability that we want to bound, which is that of $V_k(t)$, into probabilities of events involving $W(t)$ and $B^{(j)}(t)$. The probabilities on the right are squared, reflecting the fact that we have two random walks (one at time $0$ and another at time $t$) that must both stay positive. Apart from the first period, which is important to retain separately, only the even periods are included, since they are the ones on which the two random walks are mirrored.

\begin{prop}\label{VtoA'}
For any $k,n\in\N$ with $n\ge 2k$ and any $t\in[0,1]$,
\[\P\big(V_k(t) \cap E_n(t)\big) \le \P\big(A_1'(t)\cap E_n(t)\big)\cdot\prod_{j=1}^{\lfloor k/2\rfloor} \P\big(B^{(2j)}_i(t) > 0 \,\,\,\, \forall i\in[1,J_{2j}(t)] \,\big|\,  V'_{2j-1}(t) \cap E_n(t)\big)^2.\]
\end{prop}

The proof of this result involves carefully separating out as much independence as possible between the different periods and applying the FKG inequality. Again we postpone the proof to later in the section in order to continue with our overarching proof of Proposition \ref{EPhifinite}.

Next we observe that since $B^{(j)}(t)$ is simply an independent random walk started from $W_{I_{j-1}(t)-1}(t)$, it has the same distribution as $W$ itself over the $(j+1)$th period. This inspires our next proposition, which allows us to telescope the product from Proposition \ref{VtoA'} back into a statement only about $W$.

\begin{prop}\label{telescope}
For any $k,n\in\mathbb{N}$ with $n\ge 2k$ and any $t\in[0,1]$,
\[\prod_{j=1}^k \P\big(B^{(2j)}_i(t) > 0 \,\,\,\, \forall i\in[1,J_{2j}(t)] \,\big|\,  V'_{2j-1}(t) \cap E_n(t)\big) = \frac{\P\big(\bigcap_{j=1}^{k+1} A'_{2j-1}(t) \cap E_n(t)\big)}{\P(A'_1(t)\cap E_n(t))}.\]
\end{prop}

Combining Propositions \ref{VtoA'} and \ref{telescope}, and then using elementary bounds, allows us to prove the following.

\begin{prop}\label{VkcapE}
Suppose that $t\in[0,1]$ and $n\in\N$. Then for any $k\ge nt/4$, we have
\[\P\big(V_k(t) \cap E_n(t)\big) \lesssim \frac{1}{nt^{1/2}}.\]
\end{prop}

Leaving the proof of Proposition \ref{VkcapE} until later, we now observe that
\begin{align*}
\P\big(P_n(0)\cap P_n(t)\big) &= \P\big(P_n(0)\cap P_n(t)\cap E_n(t)\cap E'_n(t)\big) + \P\big(P_n(0)\cap P_n(t)\cap(E_n(t)^c\cup E'_n(t)^c)\big)\\
&\le \P\big(V_{2\lfloor nt/8\rfloor + 1}(t)\cap E_n(t)\big) + \P\big(P_n(0)\cap(E_n(t)^c\cup E'_n(t)^c)\big)\\
&= \P\big(V_{2\lfloor nt/8\rfloor + 1}(t)\cap E_n(t)\big) + \P\big(P_n(0)\big)\P\big(E_n(t)^c\cup E'_n(t)^c\big)
\end{align*}
where the last equality used the independence of $Z(0)$ and the lengths of the periods at time $t$. By Proposition \ref{VkcapE}, the first term on the last line above is at most a constant times $1/(nt^{1/2})$, and by Corollary \ref{reflcor} and Lemma \ref{Eproblarge}, the second term is at most a constant times $n^{-1/2}\exp(-\delta nt)$ for some constant $\delta>0$. Thus
\[\P\big(P_n(0)\cap P_n(t)\big) \lesssim \frac{1}{nt^{1/2}}+\frac{1}{n^{1/2}}\exp(-\delta nt)\]
and so
\[\int_0^1 \frac{\P(P_n(0) \cap P_n(t))}{t^\gamma} \d t \lesssim \frac{1}{n} \int_0^1 t^{-1/2-\gamma}\d t + \frac{1}{n^{1/2}} \int_0^1 t^{-\gamma}e^{-\delta nt} \d t.\]
For $\gamma<1/2$, the first integral on the right-hand side above is finite and the second integral (which can be approximated by integrating separately over $(0,1/n]$ and $(1/n,1)$) is of order $n^{\gamma-1}$. Therefore, for $\gamma<1/2$,
\[\int_0^1 \frac{\P(P_n(0) \cap P_n(t))}{t^\gamma} \d t \lesssim n^{-1} + n^{\gamma-3/2} \asymp n^{-1}.\]
Recalling from the start of the section that
\[\E[ \Phi^\alpha_n(\gamma)] \le \frac{2}{\P(P_n^\alpha)^2} \int_0^1 \frac{\P(P_n(0) \cap P_n(t))}{t^\gamma} \d t,\]
and from Lemma \ref{Pnalphaasymp} that for any $\alpha<1/2$,
\[\P(P_n^\alpha)\asymp \frac{1}{\sqrt n},\]
we have for $\alpha,\gamma<1/2$ that
\[\E[ \Phi^\alpha_n(\gamma)] \lesssim 1.\]
This completes the proof of Proposition \ref{EPhifinite}, subject to proving all of the intermediary results above.

Before we begin to prove these results, we will need another elementary lemma as an ingredient in the proof of Proposition \ref{VtoA'}.

\begin{lem}\label{starthalfway}
If $(S_i,\, i\ge 0)$ is a simple symmetric random walk, then for any $x,y,k\in\N$,
\[\P_x(S_i \in (0,2y) \,\,\,\, \forall i\le k) \le \P_y(S_i \in (0,2y)\,\,\,\,\forall i\le k).\]
\end{lem}

This is easily proved by induction. We include a proof later, but now proceed with the much more interesting proofs of Propositions \ref{VtoA'} and \ref{telescope}. These proofs contain the main ideas of the article.

\begin{proof}[Proof of Proposition \ref{VtoA'}]
Our first step is to move from $A_j(t)$ to $A'_j(t)$. To do so, we go via a third collection of events which we call $\tilde A_j(t)$. When $j$ is odd, let $\tilde A_j(t) = A'_j(t)$. We have already mentioned that if $j$ is odd, then
\[A_j(t) \subset A'_j(t) = \tilde A_j(t).\]
When $j$ is even, define the event
\[\tilde A_j(t) = \{Z_i(0)\in (0,2W_{I_{j-1}(t)-1}(t)) \,\,\,\, \forall i\in[I_{j-1}(t), I_j(t)-1]\}.\]
We claim that when $j$ is even, we also have $A_j(t)\subset\tilde A_j(t)$. Indeed, suppose that $j$ is even. We show that if $\omega\not\in \tilde A_j(t)$ then $\omega\not\in A_j(t)$. If $\omega\not\in\tilde A_j(t)$ then there exists $i\in[I_{j-1}(t),I_j(t)-1]$ such that either $Z_i(0)\le 0$, in which case clearly $\omega\not\in A_j(t)$, or
\[Z_i(0)\ge 2W_{I_{j-1}(t)-1}(t) = Z_{I_{j-1}(t)-1}(0) + Z_{I_{j-1}(t)-1}(t).\]
Then
\[Z_i(0)-Z_{I_{j-1}(t)-1}(0) \ge Z_{I_{j-1}(t)-1}(t),\]
so since the increments of $Z_i(t)$ are the negative of the increments of $Z_i(0)$ during even periods,
\[Z_i(t)-Z_{I_{j-1}(t)-1}(t) \le - Z_{I_{j-1}(t)-1}(t)\]
and therefore $Z_i(t)\le 0$. Thus $\omega\not\in A_j(t)$, establishing our claim. We deduce that, for any $k\in\N$,
\begin{equation}\label{AtoAtilde}
A_1(t)\cap A_2(t)\cap\ldots\cap A_k(t) \subset \tilde A_1(t)\cap \tilde A_2(t)\cap\ldots\cap \tilde A_k(t).
\end{equation}

Note that the increments of $Z_i(0)$ on even periods are independent of the whole process $W_i(t)$. Combining this fact with Lemma \ref{starthalfway}, we have
\begin{equation}\label{givenFI}
\P\big(\tilde A_1(t)\cap \tilde A_2(t)\cap\ldots\cap \tilde A_k(t) \big| \F_{I(t)}\big) \le \P\big(A'_1(t)\cap A'_2(t)\cap\ldots\cap A'_k(t) \big| \F_{I(t)}\big)
\end{equation}
for any $k\in\N$, where $\F_{I(t)} = \sigma(I_j(t),j\ge 0)$. Combining \eqref{AtoAtilde} and \eqref{givenFI} and taking expectations to remove the conditioning, for any $k\in\N$ we have
\[\P(V_k(t) \cap E_n(t))\le \P(V'_k(t)\cap E_n(t)).\]
Applying Bayes' formula and then ignoring the odd terms for $j\ge3$, we have
\begin{align}
\P\big(V_k(t) \cap E_n(t)\big) &\le \P\big(A_1'(t)\cap E_n(t)\big)\cdot\prod_{j=2}^k \P\big(A'_j(t) \,\big|\, V'_{j-1}(t) \cap E_n(t)\big)\nonumber\\
&\le \P\big(A_1'(t)\cap E_n(t)\big)\cdot\prod_{j=1}^{\lfloor k/2\rfloor} \P\big(A'_{2j}(t) \,\big|\, V'_{2j-1}(t) \cap E_n(t)\big).\label{Vtoprod1}
\end{align}

We now apply the FKG inequality \eqref{fkg2}. Recalling that
\begin{align*}
A'_{2j}(t) &= \{W_{I_{2j-1}(t)-1}(t) + B^{(2j)}_i(t) \in (0,2W_{I_{2j-1}(t)-1}(t)) \,\,\,\, \forall i\in[1,J_{2j}(t)]\}\\
&= \{W_{I_{2j-1}(t)-1}(t) + B^{(2j)}_i(t) > 0 \,\,\,\, \forall i\in[1,J_{2j}(t)]\}\\
&\hspace{35mm}\cap \{W_{I_{2j-1}(t)-1}(t) + B^{(2j)}_i(t) < 2W_{I_{2j-1}(t)-1}(t) \,\,\,\, \forall i\in[1,J_{2j}(t)]\},
\end{align*}
and noting that the two events above are increasing and decreasing respectively, we get that
\begin{align*}
\P\big(A'_{2j}(t) \,\big|\, V'_{2j-1}(t) \cap E_n(t)\big) &\le \P\big(B^{(2j)}_i(t) > 0 \,\,\,\, \forall i\in[1,J_{2j}(t)] \,\big|\,  V'_{2j-1}(t) \cap E_n(t)\big)\\
&\hspace{28mm}\cdot \P\big(B^{(2j)}_i(t) < 2W_{I_{2j-1}(t)-1}(t) \,\big|\, V'_{2j-1}(t) \cap E_n(t)\big)\\
&= \P\big(B^{(2j)}_i(t) > 0 \,\,\,\, \forall i\in[1,J_{2j}(t)] \,\big|\,  V'_{2j-1}(t) \cap E_n(t)\big)^2,
\end{align*}
where the inequality comes from \eqref{fkg2} and the equality follows from symmetry about $W_{I_{2j-1}(t)-1}(t)$ (recalling that $B^{(2j)}_0(t) = W_{I_{2j-1}(t)-1}(t)$). Substituting this into \eqref{Vtoprod1}, we have shown that
\[\P\big(V_k(t) \cap E_n(t)\big) \le \P\big(A_1'(t)\cap E_n(t)\big)\cdot\prod_{j=1}^{\lfloor k/2\rfloor} \P\big(B^{(2j)}_i(t) > 0 \,\,\,\, \forall i\in[1,J_{2j}(t)] \,\big|\,  V'_{2j-1}(t) \cap E_n(t)\big)^2\]
as required.
\end{proof}

\begin{proof}[Proof of Proposition \ref{telescope}]
We work by induction on $k$. For $k=1$, we have
\[\P\big(B^{(2)}_i(t)>0 \,\,\,\,\forall i\in[1,J_2(t)] \,\big|\, V'_1(t)\cap E_n(t)\big) = \frac{\P\big(\{B^{(2)}_i(t)>0 \,\,\,\,\forall i\in[1,J_2(t)]\} \cap A'_1(t)\cap E_n(t)\big)}{\P\big(A'_1(t)\cap E_n(t)\big)}.\]
On the event $A'_1(t)\cap E_n(t)$, the law of $(B^{(2)}_i(t))_{i\in[1,J_2(t)]}$ is identical to that of $(W_{I_2(t)-1+i}(t))_{i\in[1,J_3(t)]}$, and therefore
\[\P\big(B^{(2)}_i(t)>0 \,\,\,\,\forall i\in[1,J_2(t)] \,\big|\, V'_1(t)\cap E_n(t)\big) = \frac{\P\big(A'_3(t) \cap A'_1(t)\cap E_n(t)\big)}{\P\big(A'_1(t)\cap E_n(t)\big)},\]
establishing the claim in the case $k=1$. The general case is very similar: assuming that the claim holds for $k-1$, we have
\begin{multline*}
\prod_{j=1}^k \P\big(B^{(2j)}_i(t) > 0 \,\,\,\, \forall i\in[1,J_{2j}(t)] \,\big|\,  V'_{2j-1}(t) \cap E_n(t)\big)\\
= \frac{\P\big(\bigcap_{j=1}^k A'_{2j-1}(t)\cap E_n(t)\big)}{\P\big(A'_1(t)\cap E_n(t)\big)}\P\big(B^{(2k)}_i(t)>0 \,\,\,\,\forall i\in[1,J_{2k}(t)] \,\big|\, V'_{2k-1}(t)\cap E_n(t)\big).
\end{multline*}
Considering the last term on the right-hand side above, we note that $B^{(2k)}(t)$ is independent of $A'_{2j}(t)$ given $A'_{2j-1}(t)$ for all $j<k$, and therefore the above equals
\begin{multline*}
\frac{\P\big(\bigcap_{j=1}^k A'_{2j-1}(t)\cap E_n(t)\big)}{\P\big(A'_1(t)\cap E_n(t)\big)}\P\bigg(B^{(2k)}_i(t)>0 \,\,\,\,\forall i\in[1,J_{2k}(t)] \,\bigg|\, \bigcap_{j=1}^k A'_{2j-1}(t)\cap E_n(t)\bigg)\\
\hspace{10mm}=\frac{\P\big(\{B^{(2k)}_i(t)>0 \,\,\,\,\forall i\in[1,J_{2k}(t)]\}\cap \bigcap_{j=1}^k A'_{2j-1}(t)\cap E_n(t)\big)}{\P\big(A'_1(t)\cap E_n(t)\big)}.
\end{multline*}
Provided that $2k\le n$, on the event $\bigcap_{j=1}^k A'_{2j-1}(t)\cap E_n(t)$, the law of $(B^{(2k)}_i(t))_{i\in[1,J_{2k}(t)]}$ is identical to that of $(W_{I_{2k}(t)-1+i}(t))_{i\in[1,J_{2k+1}(t)]}$, and therefore
\[\P\bigg(\Big\{B^{(2k)}_i(t)>0 \,\,\,\,\forall i\in[1,J_{2k}(t)]\Big\}\cap \bigcap_{j=1}^k A'_{2j-1}(t)\cap E_n(t)\bigg) = \P\bigg(\bigcap_{j=1}^{k+1} A'_{2j-1}(t)\cap E_n(t)\bigg)\]
which establishes the claim for $k$, completing the proof.
\end{proof}

The proof of our third proposition in this section, Proposition \ref{VkcapE}, does not contain any major ideas; it simply combines the results above with some elementary approximations.

\begin{proof}[Proof of Proposition \ref{VkcapE}]
Combining Propositions \ref{VtoA'} and \ref{telescope}, we have
\[\P\big(V_k(t) \cap E_n(t)\big) \le \frac{\P\big(\bigcap_{j=1}^{\lfloor k/2\rfloor+1} A'_{2j-1}(t) \cap E_n(t)\big)^2}{\P(A'_1(t)\cap E_n(t))}.\]
Recalling that $A'_{2j-1}(t)$ requires that $W_i(t)$ is positive on the $(2j-1)$th period, whereas $W_i(t)$ is constant on even periods, we note that
\[\bigcap_{j=1}^{\lfloor k/2\rfloor+1} A'_{2j-1}(t) = \{W_i(t)>0 \,\,\,\, \forall i\le I_{2\lfloor k/2\rfloor+1}(t)-1\}\]
and therefore
\[\P\big(V_k(t) \cap E_n(t)\big) \le \frac{\P\big(\{W_i(t)>0 \,\,\,\, \forall i\le I_{2\lfloor k/2\rfloor+1}(t)-1\}\cap E_n(t)\big)^2}{\P(A'_1(t)\cap E_n(t))}.\]
Now, $W_i(t)$ is simply a simple symmetric random walk during odd periods, and constant on even periods. Thus the probability that it stays positive up to step $I_{2\lfloor k/2\rfloor+1}(t)-1$ is exactly the probability that a simple symmetric random walk stays positive up to step $J_1(t) + J_3(t) + \ldots + J_{2\lfloor k/2\rfloor+1}(t)-1$. We deduce that 
\begin{align*}
\P\big(V_k(t) \cap E_n(t)\big) &\le \frac{\P\big(\{Z_i(t)>0 \,\,\,\, \forall i\le J_1(t) + J_3(t) + \ldots + J_{2\lfloor k/2\rfloor+1}(t)-1\}\cap E_n(t)\big)^2}{\P(A'_1(t)\cap E_n(t))}\\
&\le \frac{\P\big(Z_i(t)>0 \,\,\,\, \forall i\le J_1(t) + J_3(t) + \ldots + J_{2\lfloor k/2\rfloor+1}(t)-1 \,\big|\, E_n(t)\big)^2}{\P\big(A'_1(t)\,\big|\, E_n(t)\big)}.
\end{align*}
On the event $E_n(t)\subset \odd$, we have
\[J_1(t)+ J_3(t) + \ldots + J_{2\lfloor nt/8\rfloor+1}(t)-1 \ge J_3(t)+ J_5(t) + \ldots + J_{2\lfloor nt/8\rfloor+1}(t)\ge n/8,\]
and therefore for any $k\ge nt/4$,
\begin{equation}\label{VcapEfinal}
\P\big(V_k(t) \cap E_n(t)\big) \le \frac{\P\big(Z_i(t)>0 \,\,\,\, \forall i\le n/8\big)^2}{\P\big(A'_1(t) \,\big|\, E_n(t)\big)} = \frac{\P\big(Z_i(0)>0 \,\,\,\, \forall i\le n/8\big)^2}{\P(A'_1(t))},
\end{equation}
where the equality holds by stationarity of $Z(t)$ and the independence of $A'_1(t)$ and $E_n(t)$ (since $E_n(t)$ only involves periods 2 and later). We know from Corollary \ref{reflcor} that
\[\P\big(Z_i(0)>0 \,\,\,\, \forall i\le n/8\big) \asymp n^{-1/2},\]
and we claim that
\[\P(A'_1(t))\gtrsim t^{1/2}.\]
To see this, note that $I_1(t)$ is independent of $Z(0)$, so
\begin{align*}
\P(A'_1(t)) &= \P(Z_i(0)>0\,\,\,\,\forall i=1,\ldots,I_1(t))\\
&\ge \P\Big(I_1(t)\le \Big\lceil\frac{4}{1-e^{-t}}\Big\rceil\Big)\P\Big(Z_i(0)>0 \,\,\,\, \forall i = 1,\ldots,\Big\lceil\frac{4}{1-e^{-t}}\Big\rceil\Big).
\end{align*}
But by Markov's inequality
\[\P\Big(I_1(t)\le \Big\lceil\frac{4}{1-e^{-t}}\Big\rceil\Big) = 1- \P\Big(I_1(t)> \Big\lceil\frac{4}{1-e^{-t}}\Big\rceil\Big) \ge 1-\frac{1-e^{-t}}{4}\E[I_1(t)] = 1 - \frac{1}{2} = \frac{1}{2};\]
and by Corollary \ref{reflcor},
\[\P\Big(Z_i(0)>0 \,\,\,\, \forall i = 1,\ldots,\Big\lceil\frac{4}{1-e^{-t}}\Big\rceil\Big) \asymp (1-e^{-t})^{1/2} \asymp t^{1/2},\]
which establishes the claim. Substituting our approximations into \eqref{VcapEfinal}, we have shown that for any $k\ge nt/4$,
\[\P\big(V_k(t) \cap E_n(t)\big) \lesssim \frac{1}{nt^{1/2}}\]
as required.
\end{proof}

We now proceed with the proofs of our minor lemmas.

\begin{proof}[Proof of Lemma \ref{Pnalphaasymp}]
Recalling that
\[P_n = \{Z_i>0 \,\,\,\,\forall i = 1,\ldots,n\} \,\,\,\,\text{ and } \,\,\,\, P^\alpha_n = \big\{Z_i \ge i^\alpha \,\,\,\,\forall i = 1,\ldots,n\big\},\]
we use the fact that $\P(P_n^\alpha) = \P(P_n^\alpha | P_n)\P(P_n)$. From Corollary \ref{reflcor} we know that $\P(P_n)\asymp n^{-1/2}$. It therefore suffices to show that $\P(P_n^\alpha) \asymp \P(P_n)$ for any $\alpha<1/2$. Fix $\alpha'\in(\alpha,1/2)$. We apply \cite[Theorem 2]{ritter:growth_RW_cond_positive}, which says that we may choose $\delta>0$ such that
\[\P(Z_i\ge \delta i^{\alpha'} \,\,\,\,\forall i=1,\ldots,n) \ge \P(P_n)/2.\]
Choose $k$ such that $\delta i^{\alpha'} \ge i^\alpha$ for all $i\ge k$. Then
\begin{align*}
\P(Z_i \ge i^\alpha \,\,\,\, \forall i=1,\ldots,n) &\ge \P(Z_i = i\,\,\,\, \forall i=1,\ldots,k; \, Z_i \ge i^\alpha\,\,\,\, \forall i=k+1,\ldots,n)\\
&\ge \P(Z_i = i\,\,\,\, \forall i=1,\ldots,k; \, Z_i \ge \delta i^{\alpha'}\,\,\,\, \forall i = k+1,\ldots,n)\\
&= 2^{-k} \P(Z_i \ge \delta(i+k)^{\alpha'}-k \,\,\,\,\forall i = 1,\ldots,n-k)\\
&\ge 2^{-k}\P(Z_i \ge \delta i^{\alpha'}\,\,\,\,\forall i = 1,\ldots,n) \ge 2^{-(k+1)}\P(P_n),
\end{align*}
which completes the proof.
\end{proof}

\begin{proof}[Proof of Lemma \ref{Eproblarge}]
We begin by considering $\odd$. In order for $\odd^c$ to occur, the sum of $\lfloor nt/8\rfloor$ independent geometric random variables of parameter $(1-e^{-t})/2$ must be smaller than $n/8$; which is equivalent to a Binomial random variable of parameters $(\lceil n/8\rceil, (1-e^{-t})/2)$ being larger than $\lfloor nt/8\rfloor$. Letting $Y$ be such a random variable, we have
\[\E[e^{(\log 2)Y}] = \Big((1+e^{-t})/2 + (1-e^{-t})\Big)^{\lceil n/8\rceil} = \Big(1+(1-e^{-t})/2\Big)^{\lceil n/8\rceil} \le (1+t/2)^{\lceil n/8\rceil} \le e^{(n/8+1)t/2},\]
so
\[\P(Y\ge \lfloor nt/8\rfloor) \le \E[e^{(\log 2)Y}]e^{-(\log 2)\lfloor nt/8\rfloor} \le e^{(n/8+1)t/2 - (\log 2)(nt/8-1)} \le 2e^{1/2}e^{-(2\log 2 - 1)nt/16}.\]
This proves the required decay for $\P(\odd^c)$, and $\P(\even)=\P(\odd)$. The proof for $\P(E'_n(t)^c)$ uses a very similar Chernoff bound, noting that $I_j(t)$ is a sum of $j$ independent Geometric random variables of parameter $(1-e^{-t})/2$.
\end{proof}

\begin{proof}[Proof of Lemma \ref{starthalfway}]
Fix $y\in\N$ and let
\[p_{x,k} = \P_x(S_i \in (0,2y) \,\,\,\, \forall i\le k).\]
We claim, by induction on $k$, that $p_{x,k}$ is non-decreasing in $x$ for $x\le y$. By symmetry this is enough to prove the lemma. Clearly the claim holds for $k=0$. For general $k$, if $x=y$ then by symmetry
\[p_{y,k+1} = \frac12 p_{y-1,k} + \frac12 p_{y+1,k} = p_{y-1,k}\]
which is larger than $p_{y-1,k+1}$ by definition. On the other hand if $x<y$, then by the induction hypothesis,
\[p_{x,k+1} = \frac12 p_{x-1,k} + \frac12 p_{x+1,k} \ge \frac12 p_{x-2,k} + \frac12 p_{x,k} = p_{x-1,k+1}.\]
This completes the proof of our final lemma in this section, and therefore the proof of Proposition \ref{EPhifinite}.
\end{proof}

\section{Proof of Proposition \ref{influenceprop}: influences of $P_n$}\label{influencesec}

In this section we give estimates on the influence of each bit $m=1,2,\ldots,n$ on the event $P_n$. Proposition \ref{influenceprop} stated that for $m=1,\ldots,n$,
\[\mathcal I_m(P_n) \asymp \frac{n-m+1}{n^{3/2}},\]
where $\mathcal I_m(P_n)$ is the probability that the $m$th bit is pivotal for $P_n$, and it will be our aim to prove this. We will keep $n$ fixed and say ``$m$ is pivotal'' as shorthand for ``$m$ is pivotal for $P_n$''.

\subsection{Translating $\mathcal I_m(P_n)$ into elementary properties of the random walk}

To reduce the amount of work we will take advantage of the fact that
\begin{equation}\label{halvework}
\mathcal I_m(P_n) = \P(m \text{ is pivotal}) = 2\P(\{m \text{ is pivotal}\}\cap P_n),
\end{equation}
which holds since the event that $m$ is pivotal is independent of the value of $X_m$:
\begin{align*}
&\P(\{m \text{ is pivotal}\}\cap P_n)\\
&= \P(\{m \text{ is pivotal}\}\cap\{X_m=1\}\cap P_n) + \P(\{m \text{ is pivotal}\}\cap\{X_m=-1\}\cap P_n)\\
&= \P(\{m \text{ is pivotal}\}\cap\{X_m=-1\}\cap P_n^c) + \P(\{m \text{ is pivotal}\}\cap\{X_m=1\}\cap P_n^c)\\
&=\P(\{m \text{ is pivotal}\}\cap P_n^c).
\end{align*}

We now write down an explicit condition for the event $\{m \text{ is pivotal}\}\cap P_n$ to occur. We claim that for $m=1,2,\ldots,n$,
\begin{equation}\label{pivcond}
\{m \text{ is pivotal}\}\cap P_n = \{Z_i>0 \,\,\,\,\forall i=1,\ldots,n\}\cap \big\{\max_{m\le i\le n} Z_i\ge 2Z_{m-1}\big\}.
\end{equation}
In words, $m$ is pivotal and $P_n$ holds if and only if $Z$ stays positive for the first $n$ steps, and hits $2Z_{m-1}$ between steps $m$ and $n$.

To see why this is true, call the path of $Z$ up to step $m-1$ the \emph{first portion} of the walk, and the path from step $m$ to step $n$ the \emph{second portion}. Of course $P_n$ entails that both portions remain positive. In order for $m$ to be pivotal, we also need that when we change the sign of the $m$th bit, and therefore reflect the second portion of the path about $Z_{m-1}$, the second portion no longer remains positive. This holds if and only if the second portion (before reflection) hits $2Z_{m-1}$. See Figure \ref{pivpic}.
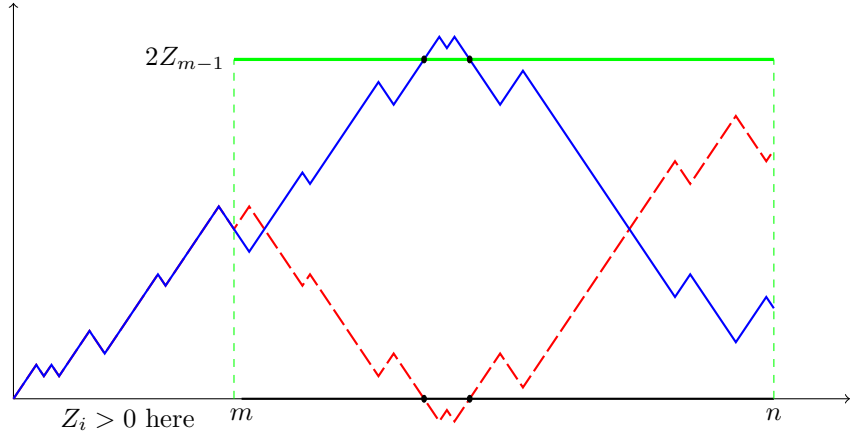
\begin{figure}[h]
\centering\begin{tikzpicture}[yscale=0.15, xscale=0.1]
\draw [<->] (0,35) -- (0,0) -- (110,0);
\draw [thick] (30,0) -- (100,0);
\draw [thick, red, dash pattern={on 7pt off 2pt}] (0,0) -- (3,3) -- (4,2) -- (5,3) -- (6,2) -- (10,6) -- (12,4) -- (19,11) -- (20,10) -- (27,17) -- (29,15) -- (31,17) -- (38,10) -- (39,11) -- (48,2) -- (50,4) -- (54,0) -- (56,-2) -- (57,-1) -- (58,-2) -- (64,4) -- (67,1) -- (87,21) -- (89,19) -- (95,25) -- (99,21) -- (100,22);
\draw [thick, blue] (0,0) -- (3,3) -- (4,2) -- (5,3) -- (6,2) -- (10,6) -- (12,4) -- (19,11) -- (20,10) -- (27,17) -- (29,15) -- (31,13) -- (38,20) -- (39,19) -- (48,28) -- (50,26) -- (54,30) -- (56,32) -- (57,31) -- (58,32) -- (64,26) -- (67,29) -- (87,9) -- (89,11) -- (95,5) -- (99,9) -- (100,8);
\node[below] at (30,0) {$m$};
\node[below] at (100,0) {$n$};
\draw[green, dashed] (29,0) -- (29,30);
\draw[green, dashed] (100,0) -- (100,30);
\draw[very thick, green] (29,30) -- (100,30);
\draw[fill] (54,0) circle [radius=0.3];
\draw[fill] (60,0) circle [radius=0.3];
\draw[fill] (54,30) circle [radius=0.3];
\draw[fill] (60,30) circle [radius=0.3];
\node[below] at (15,0) {$Z_i>0$ here};
\node[left] at (29,30) {$2Z_{m-1}$};
\end{tikzpicture}
\caption{A realisation of $Z$ with and without the $m$th bit flipped (dashed red / solid blue). The black dots show the points at which the walks hits one of the two barriers at $0$ or $2Z_{m-1}$, which is the key to pivotality.}\label{pivpic}
\end{figure}

If $m=1$ then trivially $Z_{m-1}=0$, so \eqref{pivcond} reduces to
\[\{1 \text{ is pivotal}\}\cap P_n = \{Z_i>0 \,\,\,\,\forall i=1,\ldots,n\}.\]
Thus, by Corollary \ref{reflcor}, $\P(\{1 \text{ is pivotal}\}\cap P_n)$ is of order $n^{-1/2}$. Proposition \ref{influenceprop} therefore holds for $m=1$ and we may assume that from now on $m\ge 2$.

Returning to \eqref{pivcond} in the case $m\ge 2$, the next step is to split the event that $m$ is pivotal over the possible values of $Z_{m-1}$. Writing $\P_z$ for the probability measure under which our walk starts from $z$ instead of $0$, by \eqref{halvework} and \eqref{pivcond}
\[\I_m(P_n) = 2\sum_{z=1}^{m-1} \P_0\Big(\min_{1\le i\le m-1} Z_i > 0, \, Z_{m-1} = z\Big)\cdot\P_z\Big(\big\{\min_{i \le n-m+1} Z_i > 0\big\}\cap \big\{\max_{i\le m-n+1} Z_i\ge 2z\big\}\Big).\]
By the ballot theorem \cite{andre:ballot} (or see \cite{addario_berry_reed:ballot_theorems} for a thorough introduction), the probability that a simple symmetric random walk starting from $0$ stays positive up to step $m-1$ and finishes at $z$ is $z/(m-1)$ times the probability that the random walk finishes at $z$; thus
\begin{equation}\label{Imaster}
\I_m(P_n) = 2\sum_{z=1}^{m-1} \frac{z}{m-1}\P_0(Z_{m-1} = z)\cdot\P_z\Big(\big\{\min_{i \le n-m+1} Z_i > 0\big\}\cap \big\{\max_{i\le m-n+1} Z_i\ge 2z\big\}\Big).
\end{equation}

\subsection{A lower bound on the influences of $P_n$}

Define the events
\begin{equation}\label{LandUdef}
L = L(m,n) = \big\{\min_{i \le n-m+1} Z_i > 0\big\} \,\,\,\,\text{ and }\,\,\,\, U = U(m,n,z) = \big\{\max_{i\le n-m+1} Z_i\ge 2z\big\}.
\end{equation}
Let
\[l(m,n) = \Big\lfloor \frac{\sqrt{n-m+1}}{2} \Big\rfloor \wedge \Big\lfloor \frac{\sqrt{m-1}}{2} \Big\rfloor.\]
We want to bound $\P_z(L\cap U)$ from below when $z\le l(m,n)$. The following corollary of Lemmas \ref{LCLT} and \ref{refl} will be useful.

\begin{cor} \label{Lboundcor}
If $0\le z \le \sqrt{n-m+1}$ then
\[\P_z(L(m,n)) \asymp \frac{z+1}{\sqrt{n-m+1}}\]
and if $0\le z \le l(m,n)$ then
\[\P_z(U(m,n,z)) \asymp 1.\]
\end{cor}

\begin{proof}
From Lemma \ref{refl},
\[\P_z(L) = \P_z(Z_i>0 \,\,\,\,\forall i\le n-m+1) = \P_0(Z_{n-m+1}\in [-z+1,z]),\]
and by Lemma \ref{LCLT}, this is of order
\[\sum_{i=-z+1}^z \frac{1}{\sqrt{n-m+1}}\exp\Big(-\frac{i^2}{2(n-m+1)}\Big).\]
The first part of the result now follows from the fact that $z\le \sqrt{n-m+1}$. The second part is very similar: using Lemmas \ref{refl} and \ref{LCLT},
\begin{multline*}
\P_z(U) = 1-\P_z(L) = 1-\P_0(Z_{n-m+1}\in [-z+1,z]) \ge \P_0(Z_{n-m+1}\ge z+1)\\
\ge \sum_{y=z+1}^{\lfloor\sqrt{n-m+1}\rfloor}\P_0(Z_{n-m+1} = y) \gtrsim \sum_{y=z+1}^{\lfloor\sqrt{n-m+1}\rfloor} \frac{1}{\sqrt{n-m+1}} \asymp 1
\end{multline*}
and clearly $\P_z(U)\le 1$ so the proof is complete.
\end{proof}

\begin{lem}\label{tribound}
For $z\in[0,l(m,n)]$, we have
\[\P_z\Big(L(m,n)\cap U(m,n,z)\Big) \gtrsim \frac{z}{\sqrt{n-m+1}}.\]
\end{lem}

\begin{proof}
We would like to use the FKG inequality. Unfortunately, neither $L$ nor $U$ is either increasing or decreasing as a function of $X$. However, if we replace the switch random walk $Z$ with the compass random walk $Y$, setting
\[L' = \big\{\min_{i \le n-m+1} Y_i > 0\big\} \,\,\,\,\text{ and }\,\,\,\, U' = \big\{\max_{i\le n-m+1} Y_i\ge 2z\big\},\]
then $L'$ and $U'$ are both increasing. Thus the FKG inequality \eqref{fkg1} tells us that
\[\P_z(L'\cap U')\ge \P_z(L')\P_z(U')\]
and since $Y$ and $Z$ have the same distribution,
\[\P_z(L\cap U) = \P_z(L'\cap U') \ge \P_z(L')\P_z(U') = \P_z(L)\P_z(U).\]
The result now follows from Corollary \ref{Lboundcor}.
\end{proof}

Substituting the result of Lemma \ref{tribound} into \eqref{Imaster} gives that
\begin{align*}
\mathcal I_m(P_n) &\ge 2\sum_{z=1}^{l(m,n)} \frac{z}{m-1}\P_0(Z_{m-1}=z)\cdot\P_z\Big(\big\{\min_{i \le n-m+1} Z_i > 0\big\}\cap \big\{\max_{i\le m-n+1} Z_i\ge 2z\big\}\Big)\\
&\gtrsim \sum_{z=1}^{l(m,n)} \frac{z}{m-1}\P_0(Z_{m-1}=z) \cdot\frac{z}{\sqrt{n-m+1}}.
\end{align*}
Applying Lemma \ref{LCLT} again tells us that for $z\in[1,l(m,n)]$, we have $\P_0(Z_{m-1}=z)\asymp (m-1)^{-1/2}$; so
\[\mathcal I_m(P_n) \gtrsim \sum_{z=1}^{l(m,n)} \frac{z}{m-1}\cdot \frac{1}{\sqrt{m-1}} \cdot\frac{z}{\sqrt{n-m+1}} \asymp \frac{l(m,n)^3}{(m-1)^{3/2}(n-m+1)^{1/2}}.\]
If $m\le n/2$, then the right-hand side above is of order $n^{-1/2}$, and if $m>n/2$, it is of order $(n-m+1)/n^{3/2}$. In either case this completes the proof of the lower bound in Proposition \ref{influenceprop}.

\subsection{An upper bound on the influences of $P_n$}

We will now bound \eqref{Imaster} from above. This direction is far more involved as we need to consider the entire sum; for the lower bound we could restrict to just the values of $z$ that gave the biggest contribution. We recall the definitions of $L$ and $U$ from \eqref{LandUdef}. As part of our proof we will have to bound several sums of the following form.

\begin{lem}\label{sumsub}
If $c\in\N$ and $r\ge 0$ then
\[\sum_{z=0}^\infty (z+1)^r \exp\Big(-\frac{z^2}{c}\Big) \lesssim c^{(r+1)/2}.\]
\end{lem}

\begin{proof}
Letting $C = \lceil \sqrt c\rceil$, we have
\begin{align*}
\sum_{z=0}^\infty (z+1)^r \exp\Big(-\frac{z^2}{c}\Big) &= \sum_{k=0}^\infty \sum_{z=kC}^{(k+1)C-1} (z+1)^r \exp\Big(-\frac{z^2}{c}\Big)\\
&\le \sum_{k=0}^\infty C((k+1)C)^r \exp\Big(-\frac{k^2 C^2}{c}\Big)\\
&\le C^{r+1} \sum_{k=0}^\infty (k+1)^r \exp(-k^2) \asymp C^{r+1}.\qedhere
\end{align*}
\end{proof}

Let $M=\lfloor(m-1)^{3/4}\rfloor$. We begin our upper bound on \eqref{Imaster} by splitting the sum depending on whether $z$ is larger or smaller than $M$: from \eqref{Imaster},
\begin{align}
\mathcal I_m(P_n) &= 2 \sum_{z=1}^{M} \frac{z}{m-1} \P_0(Z_{m-1} = z)  \P_z(L \cap U) + 2 \sum_{z=M+1}^{m-1} \frac{z}{m-1} \P_0(Z_{m-1} = z)  \P_z(L \cap U)\nonumber\\
&\le 2 \sum_{z = 1}^M \frac{z}{m-1} \P_0(Z_{m-1} = z)  \P_z(L\cap U) + 2 \sum_{z=M+1}^{m-1} \P_0(Z_{m-1} = z) \P_z(L).\label{twosumsI}
\end{align}
We label the two sums in \eqref{twosumsI} by (\ref{twosumsI}\,i) and (\ref{twosumsI}\,ii).

Addressing the second sum first, we note that $\P_z(L)$ is increasing in $z$, so
\[\text{(\ref{twosumsI}\,ii)} \le 2\P_{m-1}(L) \sum_{z=M+1}^{m-1} \P_0(Z_{m-1} = z) = 2\P_{m-1}(L)\P_0(Z_{m-1}>M).\]
By Lemma \ref{chernoff} with $x=M$, we have
\[\P_0(Z_{m-1}>M) \le \exp(-(m-1)^{1/2}/2).\]
If $m-1 > (n-m+1)^{1/2}$ then we use the trivial bound $\P_{m-1}(L)\le 1$, or if $m-1\le (n-m+1)^{1/2}$ then we apply Corollary \ref{Lboundcor} to obtain
\[\P_{m-1}(L) \asymp \frac{m}{\sqrt{n-m+1}}.\]
Putting these estimates together, we have shown that
\[\text{(\ref{twosumsI}\,ii)}\lesssim \Big(\frac{m}{\sqrt{n-m+1}}\wedge 1\Big)\exp(-(m-1)^{1/2}/2).\]
By considering the two cases $m<\sqrt n$ and $m\ge \sqrt n$ separately, one can check that in either case the above is at most a constant times $(n-m+1)n^{-3/2}$, as required. It thus remains to bound (\ref{twosumsI}\,i).

To do this we split it again depending on whether $z$ exceeds $\lfloor(n-m+1)^{1/2}\rfloor$. If it does not, we bound $\P_z(L\cap U)$ above by $\P_z(L)$ and apply Lemma \ref{LCLT} and Corollary \ref{Lboundcor}. Letting $M' = M\wedge \lfloor(n-m+1)^{1/2}\rfloor$, we obtain
\begin{align}
\sum_{z=1}^{M'} \frac{z}{m-1}\P_0(Z_{m-1}=z)\P_z(L\cap U) &\le \sum_{z=1}^{M'} \frac{z}{m-1} \P_0(Z_{m-1}=z)\P_z(L)\nonumber\\
&\asymp \sum_{z=1}^{M'} \frac{z}{m-1} \frac{1}{(m-1)^{1/2}}e^{-z^2/(2(m-1))} \frac{z+1}{(n-m+1)^{1/2}}.\label{zsmallUB}
\end{align}
If $m\le n/2$, then by Lemma \ref{sumsub},
\[\sum_{z=1}^{M'} z(z+1) e^{-\frac{z^2}{2(m-1)}} \lesssim (m-1)^{3/2},\]
whereas if $m>n/2$, then
\[\sum_{z=1}^{M'} z(z+1) e^{-\frac{z^2}{2(m-1)}} \le \sum_{z=1}^{\lfloor(n-m+1)^{1/2}\rfloor} z(z+1) \asymp (n-m+1)^{3/2}.\]
Applying these two bounds to \eqref{zsmallUB} gives that
\begin{equation}\label{penultimatecase}
\sum_{z=1}^{M'} \frac{z}{m-1}\P_0(Z_{m-1}=z)\P_z(L\cap U) \lesssim \frac{\big((m-1)^{3/2}\wedge (n-m+1)^{3/2}\big)}{(m-1)^{3/2}(n-m+1)^{1/2}} \lesssim \frac{n-m+1}{n^{3/2}},
\end{equation}
as required.

When $z>(n-m+1)^{1/2}$ then we bound $\P_z(L\cap U)$ above by $\P_z(U)$ instead of $\P_z(L)$. Applying Lemma \ref{LCLT}, we have
\[\sum_{z=M'+1}^{M} \frac{z}{m-1}\P_0(Z_{m-1}=z)\P_z(L\cap U) \lesssim \sum_{z=M'+1}^M \frac{z}{m-1} \frac{1}{(m-1)^{1/2}}e^{-z^2/(2(m-1))} \P_z(U),\]
and by Lemmas \ref{refl} and \ref{chernoff},
\begin{align*}
\P_z(U) &= 1-\P_z(Z_i<2z \,\,\,\,\forall i\le n-m+1)\\
&= 1 - \P(Z_{n-m+1}\in[-z+1,z]) \le 2\P(Z_{n-m+1}\ge z) \le 2\exp\Big(-\frac{z^2}{2(n-m+1)}\Big).
\end{align*}
Thus
\begin{equation}\label{onecaseremains}
\sum_{z=M'+1}^{M} \frac{z}{m-1}\P_0(Z_{m-1}=z)\P_z(L\cap U) \le \sum_{z=M'+1}^M \frac{2z}{(m-1)^{3/2}}e^{-z^2/(2(m-1)) - z^2/(2(n-m+1))}.
\end{equation}
If $m>n/2$, then the above is at most
\[\sum_{z=0}^\infty \frac{2z}{(m-1)^{3/2}}e^{- z^2/(2(n-m+1))}\]
and by Lemma \ref{sumsub}, this is of order at most $(n-m+1)/n^{3/2}$. On the other hand, if $m<n/2$ and $M'\le M$, then
\begin{align*}
\eqref{onecaseremains}&\le \sum_{z=\lfloor(n-m+1)^{1/2}\rfloor}^\infty \frac{2z}{(m-1)^{3/2}}\exp\Big(-\frac{z^2}{2(m-1)}\Big)\\
&\lesssim \frac{1}{(m-1)^{3/2}}\exp\Big(-\frac{n-m+1}{2(m-1)}\Big) \sum_{z=0}^\infty z\exp\Big(-\frac{z^2}{2(m-1)}\Big)
\end{align*}
and by Lemma \ref{sumsub}, this is of order at most
\[\frac{1}{(m-1)^{1/2}}\exp\Big(-\frac{n-m+1}{2(m-1)}\Big).\]
Since $e^{-x/2} \le x^{-1/2}$ for all $x>0$, this is bounded above by $(n-m+1)^{-1/2}$. Thus we have shown that when $M'\le M$,
\[\eqref{onecaseremains} \lesssim \frac{n-m+1}{n^{3/2}} \wedge \frac{1}{(n-m+1)^{1/2}} \le \frac{n-m+1}{n^{3/2}},\]
and of course when $M'>M$ the sum is empty and $\eqref{onecaseremains}=0$. Combining this with \eqref{penultimatecase}, we have shown that
\[\text{(\ref{twosumsI}\,i)} \lesssim \frac{n-m+1}{n^{3/2}},\]
which completes the proof of Proposition \ref{influenceprop}.

\section{Proofs of Lemmas \ref{Hdimlower}, \ref{Ttechnicality} and \ref{ergodic}}\label{techlemsec}

To complete our proof of the lower bound on the Hausdorff dimension of $\Ecal$ outlined in Section \ref{outlineproof}, we need several technical lemmas. In this section we prove those results, beginning with Lemma \ref{Hdimlower}, which is based on \cite[Lemma 6.2]{schramm_steif:noise_sens_percolation}.

\begin{proof}[Proof of Lemma \ref{Hdimlower}]
If we let $\mu_n^\alpha$ be the measure on $[0,1]$ given by
\[\mu_n^\alpha(A) = \frac{1}{\P(P_n^\alpha)} \int_A \ind_{P_n^\alpha(t)}\d t,\]
then noting that $\mu_n^\alpha$ is supported on $\bar T_n^\alpha$, \cite[Lemma 6.2]{schramm_steif:noise_sens_percolation} gives a sufficient condition for the Hausdorff dimension of $\bigcap_n \bar T_n^\alpha$ to be at least $\gamma$. This condition is that there exists a finite constant $c$ such that for infinitely many $n$,
\[\mu_n^\alpha([0,1]) \ge 1/c \,\,\,\, \text{ and } \,\,\,\, \int_0^1 \int_0^1 |t-s|^{-\gamma} \d \mu_n^\alpha(s) \d \mu_n^\alpha(t) \le c.\]
In order to prove our lemma it therefore suffices to show that this condition holds with positive probability for $\alpha<1/2$.

We start by bounding $\mu_n^\alpha([0,1])$ from below. By the Paley-Zygmund inequality,
\begin{equation}\label{paleymu}
\P\Big(\mu_n^\alpha([0,1]) \ge \frac{1}{2}\E[\mu_n^\alpha([0,1])]\Big) \ge \frac{\E[\mu_n^\alpha([0,1])]^2}{4\E[\mu_n^\alpha([0,1])^2]}.
\end{equation}
By Fubini's theorem and stationarity,
\[\E[\mu_n^\alpha([0,1])] = \frac{1}{\P(P_n^\alpha)} \int_0^1 \P(P_n^\alpha(t))\d t = \frac{1}{\P(P_n^\alpha)} \int_0^1 \P(P_n^\alpha) \d t = 1.\]
Also, for any $\gamma \in[0,1)$,
\[\E[\mu_n^\alpha([0,1])^2] = \E\Big[\int_0^1 \int_0^1 \ind_{P_n^\alpha(s)}\ind_{P_n^\alpha(t)} \d s \d t\Big] = \E[\Phi_n^\alpha(0)] \le \E[\Phi_n^\alpha(\gamma)].\]
Substituting these estimates into \eqref{paleymu}, we have
\[\P(\mu_n^\alpha([0,1])\ge 1/2) \ge \frac{1}{4\E[\Phi_n^\alpha(\gamma)]}\]
so fixing $\gamma$ to take the value in the statement of the lemma and letting $S = \sup_n \E[\Phi_n^\alpha(\gamma)]$, we have
\[\inf_n \P(\mu_n^\alpha([0,1])\ge 1/2) \ge \frac{1}{4S}.\]

Now note that
\[\Phi_n^\alpha(\gamma) = \int_0^1 \int_0^1 |t-s|^{-\gamma} \d \mu_n^\alpha(s) \d \mu_n^\alpha(t),\]
so the second part of our desired condition requires us to show that $\Phi_n^\alpha(\gamma)\le c$ for some constant $c$ and infinitely many $n$. By Markov's inequality,
\[\sup_n \P(\Phi_n^\alpha(\gamma)> 8S^2) \le \sup_n \frac{\E[\Phi_n^\alpha(\gamma)]}{8S^2} = \frac{1}{8S},\]
and therefore
\[\inf_n \P(\mu_n^\alpha([0,1])\ge 1/2 \text{ and } \Phi_n^\alpha(\gamma) \le 8S^2) \ge \inf_n \P(\mu_n^\alpha([0,1])\ge 1/2) - \sup_n \P(\Phi_n^\alpha(\gamma)> 8S^2) \ge \frac{1}{8S}.\]
By Fatou's lemma we deduce that
\[\P(\mu_n^\alpha([0,1])\ge 1/2 \text{ and } \Phi_n^\alpha(\gamma) \le 8S^2 \text{ for infinitely many } n) \ge \frac{1}{8S},\]
and the proof is complete.
\end{proof}

Our proof of Lemma \ref{Ttechnicality} is based on the equivalent result for percolation by H{\"a}ggstr{\"o}m, Peres and Steif \cite[Lemma 3.2]{haggstrom_peres_steif:dynamical_perco}.

\begin{proof}[Proof of Lemma \ref{Ttechnicality}]
Recall that for each $j$, $(N_j(t), t\ge 0)$ is a Poisson process of rate $1$ that decides when $X_j$ rerandomises. For $i\ge 0$, let $\tau^{(i)}_j = \inf\{t\ge 0 : N_j(t)=i\}$, the time of the $i$th rerandomisation of $X_j$.

Fix $i$ and $j$. Since each step of the random walk evolves (in time) independently, almost surely at time $\tau^{(i)}_j$ the random walk hits both $0$ and $2Z_{j-1}(\tau^{(i)}_j)$ after step $j$; thus for large enough $n$, the random walk hits $0$ before step $n$ regardless of the state of step $j$. The random walk therefore also falls below the line $i\mapsto i^\alpha$ before step $n$ (for large enough $n$), regardless of the state of step $j$. That is, almost surely, $\tau^{(i)}_j \not\in \bar T_n^\alpha \setminus T_n^\alpha$ for all large $n$.

However, since the system only changes when one of the $X_j$ rerandomises, for each $\alpha\ge 0$ and $n\in\N$ we have
\begin{equation}\label{barminus}
\bar T_n^\alpha \setminus T_n^\alpha \subset \{\tau^{(i)}_j : i=0,1,2,\ldots,\,\, j=1,2,\ldots,n\}.
\end{equation}

Thus for each $N$ we have
\[\bigcap_{n\ge N} (\bar T_n^\alpha \setminus T_n^\alpha) = \emptyset\,\,\,\, \text{ almost surely.}\]
However, since the $T_n^\alpha$ are nested,
\[\Big(\bigcap_{n\ge 1}\bar T_n^\alpha\Big) \setminus \Big(\bigcap_{n\ge 1} T_n^\alpha\Big) \subset \bigcup_{N\ge 1} \bigcap_{n\ge N} (\bar T_n^\alpha\setminus T_n^\alpha)\]
so the left-hand side is also empty almost surely, as required.
\end{proof}

Finally, Lemma \ref{ergodic} is a standard application of the ergodic theorem.

\begin{proof}[Proof of Lemma \ref{ergodic}]
To apply the ergodic theorem (see for example \cite[Theorem 24.1]{billingsley:probandmeas} and the surrounding chapter for further details), we should formally construct our probability space. For each $i\in\{0,1,2,\ldots\}$ and $j\in\N$ we take a Bernoulli random variable $B^{(i)}_j$ and an exponential random variable $E^{(i)}_j$ of parameter $1$. We view our space $\Omega$ as the set of sequences $(((B^{(i)}_j, E^{(i)}_j)_{i\ge 0})_{j\ge 1})$, with the product $\sigma$-algebra. We can then define $X_j(t)$ to take the value $B^{(i)}_t$ whenever $\sum_{k<i}E^{(i)}_j \le t < \sum_{k\le i} E^{(i)}_j$. We have the shift map $\theta : \Omega \to \Omega$ which maps $(((B^{(i)}_j, E^{(i)}_j)_{i\ge 0})_{j\ge 1})$ to $(((B^{(i)}_j, E^{(i)}_j)_{i\ge 0})_{j\ge 2})$; in practical terms, $\theta$ deletes $X_1(t)$ and builds our (dynamical) random walks from $(X_2(t),X_3(t),\ldots)$ instead. Standard methods show that $\theta$ is ergodic. Define
\[\Ecal'_\alpha = \Big\{t\in[0,1] : \liminf_{n\to\infty} \frac{-Z_n(t)}{n^\alpha} > 0 \Big\}.\]
For any $\alpha\ge0$, the Hausdorff dimension of $\Ecal_\alpha\cup\Ecal'_\alpha$ is invariant under $\theta$, and therefore constant almost surely by the ergodic theorem. By symmetry, the Hausdorff dimension of $\Ecal_\alpha$ equals that of $\Ecal'_\alpha$. Since the Hausdorff dimension of the union of two sets is the maximum of their Hausdorff dimensions, the Hausdorff dimension of $\Ecal_\alpha$ must therefore equal that of $\Ecal_\alpha\cup\Ecal'_\alpha$, and thus be constant almost surely.
\end{proof}

\section*{Acknowledgements}

MR would like to thank Emily Atkinson, who spent a portion of her summer internship exploring an earlier unsuccessful method to attempt to prove Theorem \ref{NSthm}, and Jon Warren for pointing out the example in \cite{warren:tanaka}. He would also like to thank the Royal Society for funding his University Research Fellowship. MP would like to thank the University of Bath for his University Research Scholarship.

\bibliographystyle{plain}
\def\cprime{$'$}

\end{document}